\documentclass[12pt]{amsart}
\usepackage[margin=2.1cm]{geometry}
1
\usepackage{mathrsfs, amssymb}
\usepackage{array}
\usepackage{float}
\usepackage{mathcomp}
\usepackage{upgreek}
\usepackage[matrix,arrow,curve,all, frame]{xy}
\usepackage{amsthm}
\usepackage{hyperref}
\usepackage{extpfeil, extarrows} 
\usepackage{tabularx,multirow,makecell,longtable,hhline}
\newcommand\ChangeRT[1]{\noalign{\hrule height #1}}
\usepackage{graphicx}
\usepackage{tikz}
\usetikzlibrary{decorations.pathreplacing}
\newcommand{\CC}{\mathbb{C}}
\newcommand{\ZZ}{\mathbb{Z}}
\newcommand{\PP}{\mathbb{P}}
\newcommand{\QQ}{\mathbb{Q}}

\newcommand{\mmm}{\mathfrak{m}}

\newcommand{\EEE}{\mathscr{E}}
\newcommand{\JJJ}{{\mathcal{J}}}
\newcommand{\III}{{\mathcal{I}}}
\newcommand{\AAA}{{\mathscr{A}}}
\newcommand{\BBB}{{\mathscr{B}}}

\newcommand{\MMM}{{\mathscr{M}}}
\newcommand{\FFF}{{\mathscr{F}}}
\newcommand{\GGG}{{\mathscr{G}}}
\newcommand{\OOO}{{\mathscr{O}}} 
\newcommand{\KKK}{{\mathscr{K}}}
 
\newcommand{\LLL}{{\mathscr{L}}} 
\newcommand{\NNN}{{\mathscr{N}}}

\newcommand{\DDD}{\mathscr{D}}

\newcommand{\p}{\operatorname{p}_{\mathrm{a}}}

\newcommand{\dd}{\mathrm{d}}

\newcommand{\dis}{\operatorname{discrep}}
\newcommand{\h}{\operatorname{h}}

\newcommand{\Sing}{\operatorname{Sing}}
\newcommand{\Def}{\operatorname{Def}}
\newcommand{\Cl}{\operatorname{Cl}}

\newcommand{\tors}{\mathrm{tors}}
\newcommand{\Pic}{\operatorname{Pic}}
\newcommand{\Bs}{\operatorname{Bs}}
\newcommand{\unit}{\operatorname{unit}}

\newcommand{\red}{{\operatorname{red}}}

\newcommand{\mult}{{\operatorname{mult}}}

\newcommand{\Supp}{{\operatorname{Supp}}}

\newcommand{\coker}{{\operatorname{coker}}}

\newcommand{\gr}{\operatorname{gr}}
\newcommand{\wt}{\operatorname{wt}}
\newcommand{\axis}{\mathrm{\mbox{-}axis}}
\newcommand{\Sat}{\operatorname{Sat}}
\newcommand{\Clsc}{\operatorname{Cl^{\operatorname{sc}}}}
\DeclareMathOperator{\totimes}{\mathbin{\tilde\otimes}}
\DeclareMathOperator{\toplus}{\mathbin{\tilde\oplus}}
\DeclareMathOperator{\btoplus}{\mathbin{\widetilde {\vcenter{\hbox{\scalebox{1.4}{$\oplus$}}}}}}

\newcommand{\Oplus}{\ensuremath{\vcenter{\hbox{\scalebox{1.5}{$\oplus$}}}}}
\newcommand{\tOplus}{\ensuremath{\vcenter{\hbox{\scalebox{1.5}{$\tilde\oplus$}}}}}

\newcommand{\ldeg}{\operatorname{\mathit{l}{\mbox{-}}\mathrm{deg}}}
\newcommand{\qldeg}{\operatorname{\mathit{ql}{\mbox{-}}\mathrm{deg}}}
\newcommand{\len}{\operatorname{len}}

\newcommand{\mumu}{{\boldsymbol{\mu}}}

\renewcommand{\emptyset}{\varnothing}

\newcommand{\typec}[1]{$\mathrm{\left(#1\right)}$}
\newcommand{\type}[1]{$\mathrm{#1}$}
\newcommand{\xref}[1]{\textup{\ref{#1}}}

\makeatletter
\@addtoreset{equation}{subsection}
\makeatother

\newcounter{number}
\renewcommand{\thenumber}{{\rm\arabic{number}${}^{\mathrm{o}}$}}
\newcommand{\no}{\refstepcounter{number}{\rm{\thenumber}}}

\theoremstyle{definition}
\newtheorem{subcase}[equation]{Subcase}

\swapnumbers
\theoremstyle{plain}
\newtheorem{theorem}[subsection]{Theorem}
\newtheorem{lemma}[subsection]{Lemma}

\newtheorem{proposition}[subsection]{Proposition}

\newtheorem{stheorem}[equation]{Theorem}

\newtheorem{corollary}[subsection]{Corollary}
\newtheorem{scorollary}[equation]{Corollary}
\newtheorem*{claim*}{Claim}
\newtheorem{sclaim}[equation]{Claim}
\newtheorem{slemma}[equation]{Lemma}
\newtheorem{sproposition}[equation]{Proposition}

\theoremstyle{definition}
\newtheorem{definition}[subsection]{Definition}
\newtheorem{convention}[subsection]{Convention}
\newtheorem{cremark}[subsection]{Comment}
\newtheorem*{definition*}{Definition}
\newtheorem{sdefinition}[equation]{Definition}
\newtheorem{example-remark}[subsection]{Remark-Example}
\newtheorem{subexample-remark}[equation]{Remark-Example}
\newtheorem{scase}[equation]{}

\newtheorem{snotation}[equation]{Notation}
\newtheorem{setup}[subsection]{Setup}
\newtheorem*{notation*}{Notation}
\newtheorem{example}[subsection]{Example}
\newtheorem*{example*}{Example}

\newtheorem{sexample}[equation]{Example}

\newtheorem{sremark}[equation]{Remark}

\title{Toward the classification of threefold extremal contractions with one-dimensional fibers}
\author{Shigefumi Mori}
\author{Yuri Prokhorov}
\address{\noindent
{\bf Shigefumi~Mori:}
\newline\noindent
Kyoto University Institute for Advanced Study,
Kyoto University, Kyoto, Japan;
\newline\noindent
Research Institute for Mathematical Sciences, Kyoto University, Kyoto, Japan; 
\newline\noindent
Chubu University Academy of Emerging Sciences, Chubu University, Aichi, Japan
}
\email{mori@kurims.kyoto-u.ac.jp}

\address{{\bf Yuri~Prokhorov:} 
\newline\noindent
Steklov Mathematical Institute of Russian Academy of Sciences, Moscow, Russian Federation;
\newline\noindent
Department of Algebra, 
Moscow State Lomonosov University, Russian Federation; 
\newline\noindent
HSE University, Russian Federation 
}
\email{prokhoro@mi-ras.ru}

\date{}

\begin{document}
\begin{abstract}
An extremal curve germ is a germ of a threefold $X$ with terminal singularities along a connected reduced
complete curve~$C$ such that there exists a $K_X$-negative contraction $f : X \to Z$ with~$C$ being a fiber. 
We give a rough classification of extremal curve germs with reducible central curve~$C$.
\end{abstract}
\maketitle
\tableofcontents

\section{Introduction}
Let $(X, C)$ be a germ of an analytic threefold with terminal singularities along a reduced complete connected curve $C$.
We say that $(X, C)$ is an \textit{extremal curve germ} if there is a contraction \mbox{$f : X\to Z \ni o$} such that 
$C = f^{-1}(o)_{\red}$ and $-K_X$ is $f$-ample. According to the dimension of the exceptional locus, extremal curve germ $(X, C)$ can be related to one of three classes: it is said to be 
\textit{flipping} if the corresponding contraction $f : X\to Z \ni o$ is birational and 
its exceptional locus is one-dimensional, $(X, C)$  is said to be \textit{divisorial} if $f$ is birational and 
the exceptional locus is two-dimensional, and, finally,  $(X, C)$  is said to be 
a \textit{$\QQ$-conic bundle germ} if the target of the contraction $f$ is a surface and all its fibers are 
one-dimensional. Note that the exceptional locus of a divisorial extremal curve germ 
is not necessarily equi-dimensional and therefore the target $Z \ni o$ is not necessarily $\QQ$-Gorenstein (cf. \cite[Theorem~3.1]{MP:IA}).

For birational geometry extremal curve germs are important
because they can be regarded
as building blocks in the three-dimensional minimal model program.
In the case where $C$ is irreducible, there is a classification of extremal curve germs, see \cite{KM92} and \cite{MP:cb1}:
they are divided into several classes: \typec{III}, \typec{k1A}, \typec{k2A}, \typec{cD/2}, \typec{cAx/2}, \typec{cE/2}, \typec{cD/3}, \typec{IIA}, \typec{II^\vee}, \typec{IE^\vee}, \typec{ID^\vee}, 
\typec{IC}, 
\typec{IIB}, \typec{kAD}, and \typec{k3A}.
We note that \typec{III} contains smooth ones, and
that \typec{k1A} contains \typec{IA^\vee} (cf. Remark~\ref{rem:cla}).

In the present paper we give a rough classification of extremal curve germs with reducible central curve $C$. 
In this case $(X,C_i)$ is also an
extremal curve germ for any component $C_i\subset C$.
Thus the main problem here is to glue together various $(X,C_i)$'s. 
It turns out that this is not always possible and there are a lot of natural restrictions.
Our main result is the following theorem.

\begin{theorem}
\label{thm:main}
Let $(X,C)$ be an extremal curve germ with reducible central curve
\[
C=\bigcup_{i=1}^{N} C_i,\quad N\ge 2,
\]
where the $C_i$ are irreducible components.
Then
one of the cases in Table~\xref{tab:table} below is possible, 
where $\Sing^{\mathrm{nG}}(X)$
is the set of non-Gorenstein points of $X$, the notation $n\times(\mathrm{\textasteriskcentered})$ means that 
$(X,C)$ has exactly $n$ components of type \typec{\textasteriskcentered}, and
the notation $(\textasteriskcentered)+n\times(\mathrm{\textasteriskcentered\textasteriskcentered})$ means that 
$(X,C)$ has a component of type \typec{\textasteriskcentered} and $n$ components of type \typec{\textasteriskcentered\textasteriskcentered}.
The column $f$ describes the nature of the corresponding contraction $f:X\to Z$: flipping \textup{(f)}, divisorial \textup{(d)} or $\QQ$-conic bundle \textup{(cb)}.

All the cases in the table except possibly for~\ref{thm:main:k3A} do occur.

{\rm
\setlength{\tabcolsep}{12pt}
\renewcommand{\arraystretch}{1.6}
\begin{longtable}{p{0.007\textwidth}|p{0.36\textwidth}|p{0.039\textwidth}|p{0.041\textwidth}|p{0.27\textwidth}
}
\ChangeRT{1pt}
& $(X,C)$ &$f$ &$N$& $\Sing^{\mathrm{nG}}(X)$
\\\ChangeRT{1.pt}
\endfirsthead
\ChangeRT{1pt}
& \multicolumn{2}{c}{$(X,C)$} && $\Sing^{\mathrm{nG}}(X)$
\\\ChangeRT{1.pt}
\endhead
\multicolumn{5}{c}{$X$ has at most one non-Gorenstein point}
\\\ChangeRT{1pt}
\no 
\label{thm:main:Gor}&
$X$ is Gorenstein & cb &2& $\varnothing$
\\\hline 
\no 
\label{thm:main:cD/2-cE/2}&
$2{\times}(\mathrm{cAx/2})$, $2{\times}(\mathrm{cD/2})$ or $2{\times}(\mathrm{cE/2})$
& cb
&2 &
\type{cAx/2}, \type{cD/2} or \type{cE/2}
\\\hline 
&
&f&2&
\\\cline{3-4}
\no
\label{thm:main:cD/3}&$N{\times}(\mathrm{cD/3})$&d&$\le 4$&\type{cD/3}
\\\cline{3-4}
&&cb&$\le 5$&
\\\hline
&&f&$\le 4$&
\\\cline{3-4}
\no
\label{thm:main:IIA}& 
$N{\times}(\mathrm{IIA})$
&d
&$\le 7$ &\type{cAx/4}
\\\cline{3-4}
&&cb&$\le 7$&
\\\hline 
\no 
\label{thm:main:IIdual2}&
$2{\times}(\mathrm{II^\vee})$
&cb&2& \type{cAx/4}
\\\hline
\no
\label{thm:main:IIdual}&
\multirow{3}{*}{$(\mathrm{II^\vee})+(N-1){\times}(\mathrm{IIA})$}
& f& $2$ &\multirow{3}{*}{\type{cAx/4}}
\\\cline{3-4}
&& d& $\le 4$ &
\\\cline{3-4}
&&cb&$\le 5$&
\\\hline
\no
\label{thm:main:IIB}& 
$(\mathrm{IIB})+(N-1){\times}(\mathrm{IIA})$
&
d&2 &
\multirow{2}{*}{\type{cAx/4}}
\\\cline{3-4}
&&cb&$\le 3$&
\\\hline 
\no 
\label{thm:main:IC} &
\multirow{3}{*}{$(\mathrm{IC})+(N-1){\times}(\mathrm{k1A})$}
&f
&$2$& $\frac 1m(2,m-2,1)$
\\\cline{3-4}
&&d&$\le 4$&$m$ is odd, $m\ge 5$
\\\cline{3-4}
&&cb&$\le 5$&
\\\hline 
\no 
\label{thm:main:k1A}&
$N{\times}(\mathrm{k1A})$ &f,d,cb& &\type{cA/m}
\\\hline
\multicolumn{5}{c}{$X$ has exactly two non-Gorenstein points}
\\\hline 
\no 
\label{thm:main:k3A}&
\multirow{2}{*}{$(\mathrm{k3A})+(N-1){\times}(\mathrm{k1A})$} &d
&2& 1) $\frac 1{2k-1}(1,-1, k)$
\\\cline{3-4}
&&cb&$\le 3$&2) $\frac12 (1,1,1)$
\\\hline 
\no 
\label{thm:main:kAD}&
$(\mathrm{kAD})+(N-1){\times}(\mathrm{\textasteriskcentered})$,
&f&2 &
1) $\frac 1{2k-1}(1,-1,k)$
\\\cline{3-4}
&where
each $(\textasteriskcentered)$ is either&d&$\le 4$&2) \type{cA/2}, \type{cAx/2}, or~\type{cD/2}
\\\cline{3-4}
&\typec{k1A}, \typec{cD/2} or \typec{cAx/2}&cb&$\le 5$&
\\\hline
\multicolumn{5}{c}{$X$ has two or more non-Gorenstein points}
\\\hline 
\no 
\label{thm:main:k2A+k1A}&
$n{\times}(\mathrm{k2A})+k{\times}(\mathrm{k1A})$, $n\ge1$
&&
\\\hline
\caption{\mbox{Extremal curve germs with reducible central fiber}}
\label{tab:table}
\end{longtable}
}
\end{theorem}

Thus our table provides a satisfactory description of extremal curve germs with 
reducible central fiber
in all cases except for the case~\ref{thm:main:k2A+k1A}.
Note however that in many cases the estimate of the number $N$ of components of 
the central curve $C$
seems to be not optimal. 

The following important fact is a special case of M.~Reid's General 
Elephant Conjecture; it is a direct consequence of 
\cite[Theorem~1.3]{MP20} and Theorem~\ref{thm:main}.

\begin{corollary}
\label{cor}
Let $(X,C)$ be an extremal curve germ.
Assume that $(X,C)$ is not as in~\ref{thm:main:k2A+k1A}.
Then a general member of $|-K_X|$ is normal and has only Du Val singularities.
\end{corollary}

It is expected that this fact is valid also for extremal curve germs of type~\ref{thm:main:k2A+k1A}.

The proof of Theorem~1.3 in \cite{MP20} shows that the cases in the table can 
be naturally divided in classes according to complexity of 
a general
member $D \in |-K_X |$:
\begin{itemize}
\item 
in the case \xref{thm:main:Gor} the linear system $|-K_X|$ is base point 
free; 
\item 
in the cases 
\xref{thm:main:cD/2-cE/2} -- \xref{thm:main:IIdual},
\xref{thm:main:k1A} the divisor $D$ does not contain any 
component of $C$, hence $D\cap C$ is the unique 
non-Gorenstein point $P\in X$ and $D$ is just a general 
anticanonical member of the singularity germ $(X\ni P)$ (see 
\cite[{(6.3)}]{Reid:YPG}, \cite[(0.4.14)]{Mori:flip});
\item 
in the cases \xref{thm:main:IIB}, \xref{thm:main:IC}, 
\xref{thm:main:k3A}--\ref{thm:main:k2A+k1A} \ $\dim \Bs |-K_X|=1$ and
$D$ 
contains all the components 
which are of types \typec{IIB}, \typec{IC}, \typec{k2A}, \typec{k3A}, and 
\typec{kAD}.
\end{itemize}

The structure of the paper is as follows. Sect.~\ref{sect:pre} contains some preliminary material.
In Sect.~\ref{sect:IC} we disprove the case of extremal curve germs 
of type~\typec{IC}+ \typec{k2A}
and in Sect.~\ref{sect:kAD-k3A} we do the same for types
\typec{k2A}+\typec{kAD} and \typec{k2A}+\typec{k3A}.
The main theorem is proved in Sect.~\ref{sect:proof}.
Examples given in Sect.~\ref{sect:ex} show that all the cases in the table except possibly~\ref{thm:main:k3A}
occur.
Appendix~\ref{sect:app} contains some general facts about extremal curve germs.

\subsection*{Acknowledgements}
This work was supported by the Kyoto University Institute for Advanced Study, and the Research Institute for Mathematical Sciences (RIMS), an International Joint Usage/Research Center located in 
Kyoto University.
In particular, this paper was conceived during the second author's visit to RIMS, Kyoto University in February 2020. The authors are very grateful to the institute for the support and hospitality.
The second author was also partially supported by the HSE University Basic Research Program.
The authors would like to thank the referee for careful reading the manuscript and 
many comments that help to improve the exposition.

\section{Preliminaries}
\label{sect:pre}
\begin{definition}
\label{def1}
Let $(X,\, C)$ be the analytic germ 
of a threefold with terminal singularities along a reduced connected complete curve. We 
say that $(X,\, C)$ is an \emph{extremal curve germ} if there is a contraction 
\[
f: (X,\, C)\xlongrightarrow{\hspace*{17pt}}(Z,o) 
\]
such that $C=f^{-1}(o)_{\red}$ and $-K_X$ is $f$-ample. 
Furthermore, $f$ is called \emph{flipping} if its exceptional locus coincides 
with $C$ and \emph{divisorial} if its exceptional locus is two-dimensional.
If $f$ is not birational, then $Z$ is a surface and $(X,\, C)$ is said to be a 
\emph{$\QQ$-conic bundle germ} \cite{MP:cb1}.
\end{definition}

\subsection{Notation}
For the classification of threefold terminal singularities we refer to \cite{Mori:sing} and \cite{Reid:YPG}.
Throughout the paper the techniques of \cite{Mori:flip} will be used freely. 
Below for convenience of the reader we recall only very basic definitions and facts. 

\begin{scase}
\label{Term-sing}
Let $(X,\, P)$ be a threefold terminal singularity $(X,\, P)$ of index $m\ge 0$.
Then we have 
\begin{equation}
\label{eq:Term-sing}
\uppi_1(X\setminus \{P\})\simeq \Clsc(X,P)\simeq \ZZ/m\ZZ.
\end{equation} 
(see \cite[Lemma~5.1]{Kaw:Crep})
By $(X^\sharp, P^\sharp)$ we denote its \emph{index-$1$ cover} \cite[(3.5-3.6)]{Reid:YPG}.
For any object $V$ on $X$ by $V^\sharp$ we denote its pull-back on~$X^\sharp$. 
\end{scase}

\begin{scase}(see \cite[\S 2]{Mori:flip})
Let $X$ be an analytic threefold with terminal singularities and let $C\subset X$ be a connected reduced curve. 
For coherent $\OOO_X$-modules $\GGG\supset \FFF$ the \emph{saturation} of $\FFF$ in $\GGG$ is the smallest 
$\OOO_X$-submodule $\Sat_{\GGG}(\FFF)$ containing $\FFF$ such that $\GGG/\Sat_{\GGG}(\FFF)$ has no submodules of finite length.
Let $\III_C\subset \OOO_X$ be the ideal sheaf of $C$. Then 
$\III_C^{(n)}$ denotes its \emph{symbolic $n$th power}, that is, $\Sat_{\OOO} \III_C^n$. 
For a coherent $\OOO_X$-module $\FFF$, set 
\begin{equation*}
F^n \FFF:=\Sat_{\FFF}\big(\III_C^{(n)}\FFF\big)=\Sat_{\FFF} \big(\III_C^{n}\FFF\big). 
\end{equation*} 
Further, set 
\begin{equation*}
\gr_C^n\OOO:=\III_C^{(n)}/\III_C^{(n+1)},\qquad \gr_C^n\upomega:=F^n\upomega_X/F^{n+1}\upomega_X. 
\end{equation*}
\end{scase}
\begin{scase}
Let $C\subset (X,\, P)$ be a germ of a smooth curve. 
Assume that $C^\sharp$ is also smooth. Define
\begin{equation*}
\ell(P):=\len_P \III_C^{\sharp (2)}/\III_C^{\sharp 2},
\end{equation*}
where $\III_C^\sharp$ is the ideal sheaf of $C^\sharp$ in $X^\sharp$. 
If we write the equation of $X^\sharp$ in $\CC^4$ in the form 
\[
\phi\equiv x_4^rx_1 \mod (x_1,\, x_2,\, x_3)^2
\]
so that $C^\sharp$ is the $x_4\axis$,
then by \cite[2.16]{Mori:flip} we have
\begin{equation}
\label{eq:ell(P)}
\ell(P)=r.
\end{equation}
\end{scase}

\subsection{}
Let $(X,\, C)$ be an extremal curve germ (here $C$ is possibly reducible).
Let $P\in C\subset X$ be a point of index $m\ge 1$.

\begin{sdefinition}[{\cite[(8.4)]{Mori:flip}}]
Let $\JJJ$ be a $C$-laminal ideal of pure width $d$
(cf. \cite[(8.2)]{Mori:flip}). We say that $\JJJ$ is a \emph{nested complete
intersection} (nested c.i., in short) at~$P$ if we have $(u, v)_P = \III_{C,P}$ and $(u^d,\, v)_P=\JJJ_P$ for some $u,\, v\in \OOO_{X,P}$, 
and that $\JJJ$ is \emph{locally a nest c.i. on a subset} if $\JJJ$ is a nested c.i.
at each point of the set. 
\end{sdefinition} 

\begin{sdefinition}[{\cite[(8.10)]{Mori:flip}}]
\label{def:232}
Let $\JJJ$ be a $C$-laminal ideal of pure width $d$, and $\JJJ^\sharp$ the canonical lifting
of $\JJJ$ at~$P$. If an $\ell$-basis $\{s_1,\, s_2\}$ of $\III_{C,P}$ satisfies 
$(s_1^d,\,s_2)_P = \JJJ_{P^\sharp}^\sharp$,
we say that the ordered $\ell$-basis $(s_1, s_2)$ is a \emph{$(1, d)$-monomializing $\ell$-basis} of
$\III_C\supset \JJJ$ at~$P$, and that $\JJJ$ is \emph{$(1, d)$-monomializable} at~$P$ if such $(s_1,s_2)$ exists.
We also say that $\JJJ$ is \emph{$(1, d)$-monomializable on a subset} if $\JJJ$ is $(1, d)$-monomializable 
at each point of the set.
We note that $\JJJ$ is $(1,d)$-monomializable at a Gorenstein point $P$ if and only if it is nested c.i.
\end{sdefinition}

\begin{sproposition}[{cf. \cite[(2.6)]{KM92}}]
\label{extention-class}
Let $\LLL$, $\MMM$, $\NNN$ be $\ell$-locally free $\OOO_C$-modules fitting in the $\ell$-exact sequence:
\begin{equation*}
0 \xlongrightarrow{\hspace*{17pt}} \LLL \xlongrightarrow{\hspace*{17pt}} \MMM \xlongrightarrow{\hspace*{17pt}} \NNN \xlongrightarrow{\hspace*{17pt}} 0.
\leqno{\mathrm{(E)}}
\end{equation*}
Then the extension class $[\mathrm{E}]$ belongs to $H^1(C,\, \NNN^{\totimes(-1)}\totimes \LLL)$, 
and $(\mathrm{E})$ is $\ell$-split if and only if $[\mathrm{E}]=0$.
\end{sproposition}

\begin{sproposition}
\label{thicken-I}
Assume that $C$ is c.i. at every Gorenstein point
and so is $C^{\sharp}$ at every non-Gorenstein point.
Further, assume that $\gr^1_C\OOO$
fits in an $\ell$-exact sequence
\[
0 \longrightarrow\AAA\longrightarrow\gr^1_C\OOO\longrightarrow\BBB\longrightarrow0
\]
for some $\ell$-invertible sheaves $\AAA, \BBB$.
Let $\JJJ$ be such that $\III_C^{(2)} \subset \JJJ \subset \III_C$ and $\JJJ/\III_C^{(2)} =\AAA$. Then $\JJJ$ is
$(1,2)$-monomializable at~$P$. We have 
\begin{equation*}
\AAA= \ker\left[\beta_\JJJ:\gr_C^1\OOO \xlongrightarrow{\hspace*{17pt}} \gr^1(\OOO,\, \JJJ)\right],\qquad \BBB=\gr^1(\OOO,\, \JJJ), 
\end{equation*}
and
an $\ell$-exact sequence
\begin{equation*}
0 \xlongrightarrow{\hspace*{17pt}} \BBB^{\totimes 2}\xlongrightarrow{\hspace*{17pt}} \gr^2(\OOO,\, \JJJ)\xlongrightarrow{\hspace*{17pt}}\AAA \xlongrightarrow{\hspace*{17pt}} 0.\leqno{\mathrm{E}(\JJJ, 
2)} 
\end{equation*}
\end{sproposition}

\begin{proof}[Sketch of the proof]
Let $P$ be a non-Gorenstein point of $C$ and, using the notation of~\ref{def:232}, one can choose generators $s_1, s_2$ of 
$\III^\sharp_{C}$   at~$P^\sharp$ such that $s_1$ and $s_2$ generate at $P^\sharp$ the invertible sheaves 
$\BBB^\sharp=\III^{\sharp}_{C}/\JJJ^{\sharp}$
and $\AAA^\sharp = \JJJ^{\sharp}/(\III^{\sharp}_{C})^{2}$, respectively.
Then $\III^{\sharp}_{C} = (s_1)+\JJJ^{\sharp}$ at $P^\sharp$, and $\JJJ^{\sharp}=(s_1^2, s_2)$ at $P^\sharp$ follows from 
$\JJJ^\sharp = (s_2)+(\III^\sharp_{C})^2 = 
(s_1^2, s_2)+ \III^\sharp_{C}\JJJ^\sharp$ at $P^\sharp$
by Nakayama's Lemma.
Hence $\JJJ$ is $(1,2)$-monomializable at~$P$.
Now the assertion of~\ref{thicken-I} follows from \cite[(8.10)~(i)]{Mori:flip}. The case of Gorenstein $P$ is treated similarly.
\end{proof}

\begin{sproposition}
\label{thicken-J}
Let $d > 1$.
Assume that $\JJJ$ is $(1,d)$-monomializable at every point of $C$ with 
\begin{equation}
\label{eq:JJJ-1d}
\AAA:= \ker\left[\beta_\JJJ:\gr_C^1\OOO \xlongrightarrow{\hspace*{17pt}} \gr^1(\OOO,\, \JJJ)\right],\qquad\BBB:=\gr^1(\OOO,\, \JJJ),
\end{equation}
and a standard $\ell$-exact sequence 
\begin{equation*}
0 \xlongrightarrow{\hspace*{17pt}} \BBB^{\totimes d}\xlongrightarrow{\hspace*{17pt}} \gr^d(\OOO,\, \JJJ) \xlongrightarrow{\hspace*{17pt}}\AAA \xlongrightarrow{\hspace*{17pt}} 0.
\leqno{\mathrm{E}(\JJJ, d)}
\end{equation*}
Further, assume that it is $\ell$-split, that is, $\gr^d(\OOO,\, \JJJ)=\BBB^{\totimes d}\toplus \AAA$ on $C$.
Let $\KKK$ be such that $F^{d+1}(\OOO,\, \JJJ)\subset \KKK \subset F^d(\OOO,\, \JJJ)$ and $\KKK/F^{d+1}(\OOO,\, \JJJ) =\AAA$. 
Then $\KKK$ is $(1,d+1)$-monomializable at every point of $C$. The sheaves $\AAA$ and $\BBB$ defined earlier remain unchanged
\textup(cf.~\eqref{eq:JJJ-1d}\textup)
and there is an $\ell$-exact sequence
\begin{equation*}
0 \xlongrightarrow{\hspace*{17pt}}\BBB^{\totimes d+1} \xlongrightarrow{\hspace*{17pt}} \gr^{d+1}(\OOO,\, \KKK)\xlongrightarrow{\hspace*{17pt}} \AAA \xlongrightarrow{\hspace*{17pt}} 0. 
\leqno{\mathrm{E}(\KKK,{d+1})} 
\end{equation*}
\end{sproposition}

\begin{proof}[Sketch of the proof]
Let $P$ be a non-Gorenstein point of $C$, and 
one can choose generators $s_1, s_2$ of 
$\III^\sharp_{C}$ at~$P^\sharp$ such that 
$(s_1^d, s_2) = (\JJJ^\sharp)_{P^\sharp}$. 
Then 
$s_1$ and $s_2$ generate  at $P^\sharp$ the invertible sheaves $\BBB^\sharp=\III^{\sharp}_{C}/\JJJ^{\sharp}$ and $\AAA^\sharp$, respectively, and
$\mathrm{E}(\JJJ,d)$ is $\ell$-exact (cf.~\cite[(8.10)~(i)]{Mori:flip}).
By $(F^{d+1}(\OOO, \JJJ)^\sharp)_{P^\sharp} = 
(s_1^{d+1}, s_1 s_2, s_2^2)$ (cf.~\cite[(8.3)~(i)]{Mori:flip})
and $(\KKK^\sharp)_{P^\sharp} = F^{d+1}(\OOO, \JJJ)^\sharp_{P^\sharp} + (s_2)$, one sees
$(\KKK^\sharp)_{P^\sharp} = (s_1^{d+1}, s_2)$ and that $\KKK$ is $(1, d+1)$-monomializable at~$P$. The $\ell$-exactness of the naturally induced $\mathrm{E}(\KKK, d+1)$ follows 
from~\cite[(8.10)~(i)]{Mori:flip}. The case of Gorenstein $P$ is treated similarly.
\end{proof}

\begin{stheorem}[{\cite[(8.12)]{Mori:flip}}]
\label{M88-8.12}
Let $d>1$. Assume that $\JJJ (\subset \III_C)$ is $(1,d)$-monomializable at every point of $C$. 
Then, with $\ldeg_C$ defined in \cite[(8.9.1.i)]{Mori:flip} as $l\deg_C$, we have
\begin{equation*}
\ldeg_C(\BBB) + \textstyle\frac{1}{d}\, \ldeg_C(\AAA) \ge 0.
\end{equation*}
If furthermore $(X,C)$ is an extremal contraction of birational type, then 
the inequality is strict.
\end{stheorem} 

\begin{proof} 
By the definition $\gr^n(\OOO,\, \JJJ)=F^n(\OOO,\, \JJJ)/F^{n+1}(\OOO,\, \JJJ)$
and the vanishing 
\[
H^1(X,\, \OOO_X/F^n(\OOO,\, \JJJ))=0
\]
for every $n >1$ \cite[1.2]{Mori:flip}, we have $\sum_{i=1}^{n}\chi(\gr^i(\OOO,\, \JJJ))\ge 0$
(cf. \cite[Proof of 2.3]{Mori:flip}). 
By the computation of \cite[(8.12)]{Mori:flip}, we have
\begin{equation*}
\textstyle
\chi(\gr^n(\OOO,\, \JJJ)) \le 
\frac{n^2}{2d}\left(\ldeg_C(\BBB)+ \frac{1}{d}\, \ldeg_C(\AAA) \right)+O(n).
\end{equation*}
Whence, the required inequality. 
Note that \cite[(8.12)]{Mori:flip} treats birational extremal curve germs, but 
it uses the vanishing of $H^1(\OOO_X)$ but not the vanishing of $H^1(\upomega_X)$, hence it works for $\QQ$-conic bundle cases as well.
For the last assertion we note that \cite[Corollary~2.12.10]{KM92} together with \cite[Remark~2 and Lemma~3]{Mori:err}
imply that 
\[
\chi\left(\OOO_X / F^n(\OOO,J)\right)\ge O(n^3)\quad \text{as $n$ grows.}\qedhere
\]
\end{proof} 

\begin{slemma}
Let $f: (X,\, C)\to (Z,o)$ be an extremal curve germ with $C=C_1\cup C_2$,
where $C_1$ and $C_2$ are irreducible. 
If $f$ is a $\QQ$-conic bundle, we additionally assume that $(Z,o)$ is smooth
\textup(cf. Corollary~\xref{cor:divI}\xref{cor:divIb} and Convention~\xref{conve:Qcb}\textup) and 
$X$ is not Gorenstein. 
Then there is an isomorphism 
\begin{equation}
\label{eq:refz:xic.3.6.1}
\gr^0_C \upomega \simeq\gr^0_{C_1} \upomega \oplus\gr^0_{C_2} \upomega
\end{equation}
of sheaves. 
\end{slemma}

\begin{proof}
If $f$ is birational, then the assertion follows from \cite[(1.14)(ii)]{Mori:flip}.
So, we assume that $f$ is a $\QQ$-conic bundle. We have an injection 
\begin{equation*}
\gr^0_C \upomega \xhookrightarrow{\hspace{17pt}} \gr^0_{C_1} \upomega \oplus\gr^0_{C_2} \upomega
\end{equation*}
whose cokernel $\FFF$ is of finite length. Note that $\gr^0_{C_i} \upomega\simeq \OOO_{C_i}(-1)$ \cite[(1.14)(i)]{Mori:flip}.
Hence, $\len(\FFF)=-\chi(\gr^0_C \upomega)$ and it is sufficient to show that $H^1(\gr^0_C \upomega)=0$.
Assume the converse. 
It follows from \cite[Corollary~4.4.1]{MP:cb1} that $C$ coincides with the whole scheme fiber $f^{-1}(o)$.
But then $-K_X\cdot C=2$. On the other hand, $-K_X\cdot C_i\le 1$ and this inequality is strict if $X$ is not Gorenstein along $C_i$ \cite[(0.4.10.3)]{Mori:flip},
a contradiction.
\end{proof}

\begin{scorollary}
In the above notation, if the point $P=C_1\cap C_2$ is of index $m$, then at~$P^\sharp$ we have
\begin{equation}
\label{eq:ref(z:xic.3.6).2}
\bigl(
\upomega_{X^{\sharp}} \otimes \OOO/(\III^{\sharp}_{C_1}+\III^{\sharp}_{C_2})
\bigr)^{\mumu_m}=0.
\end{equation}
\end{scorollary}
\begin{proof}
We have an exact sequence
\begin{equation*}
0 \longrightarrow
\upomega_{X^{\sharp}}\otimes \OOO/\III_{C^{\sharp}}
\xlongrightarrow{\hspace*{7pt}\varphi \hspace*{7pt}}
\upomega_{X^{\sharp}}\otimes \bigl(\OOO/\III_{C_1}^{\sharp} \oplus \OOO/\III_{C_2}^{\sharp}\bigr)
\xlongrightarrow{\hspace*{7pt}}
\upomega_{X^{\sharp}}\otimes \OOO/\bigl(\III_{C_1}^{\sharp}+\III_{C_2}^{\sharp}\bigr)
\longrightarrow
0
\end{equation*}
on the index 1 cover $X^{\sharp}$ (e.g. $\III_{C}^{\sharp}$ denotes the defining ideal
of $C^{\sharp}$ in $X^{\sharp}$) and the $\mumu_m$-invariant part of
$\varphi$ is \eqref{eq:refz:xic.3.6.1} at~$P$. Thus \eqref{eq:refz:xic.3.6.1}
implies \eqref{eq:ref(z:xic.3.6).2}.
\end{proof} 

\begin{convention}
\label{conve:Qcb}
In Sections~\ref{sect:IC} and~\ref{sect:kAD-k3A} of this paper we consider extremal curve germs $(X,C)$ with reducible $C$ 
which contain a component of type~\typec{IC}, \typec{kAD} or \typec{k3A}.
So in the case of $\QQ$-conic bundles we may assume that the base is smooth in these cases.
Indeed, $\QQ$-conic bundles over singular base have been classified in \cite[Theorem~1.3]{MP:cb2}
and according to this classification each component $C_i\subset C$ is locally imprimitive of type 
\typec{II^\vee}
or \typec{k1A} (see~\ref{thm:main:IIdual2} and~\ref{thm:main:k1A}). 
\end{convention}

\section{Case \typec{IC}+ \typec{k2A}}
\label{sect:IC}
The purpose of this section is to disprove the following situation.

\begin{setup}
\label{setup:IC}
Let $(X,\, C)$ be an extremal curve germ such
that $C = C_1 \cup C_2$, where $C_1$ and $C_2$ are 
irreducible, $C_1$ is of type~\typec{k2A}, and $C_2$ is
of type~\typec{IC}. Let $C_1 \cap C_2=\{P\}$. Then $P \in C_2$ is the \typec{IC} point
of index $m$ (odd and $m \ge 5$).
\end{setup} 

\begin{theorem}
\label{thm:IC}
The case \xref{setup:IC} does not
occur.
\end{theorem}

\subsection{}
\label{z:xic.3}
The proof will be done through a series of lemmas.
We know that $X$ has a point $R\neq P$ of index $m'>1$ on $C_1$ and is 
smooth outside $\{P,\, R\}$. 

We note that \cite[(A.3)]{KM92} applies to the curve germ $(X,\, C_2)$,
since birational extremal curve germs of type~\typec{IC} are flipping \cite[(8.3.3)]{KM92}.

\begin{lemma}[Deformation at $R$]
\label{lemma:xic.3.1}
We may assume $R$ is ordinary, that is, $(X,\, R)$ is a 
cyclic quotient singularity. In particular by \cite[(9.3)]{Mori:flip},
there is a coordinate system
\begin{equation*}
(X,\, R) =\CC^3_{z_1,\, z_2,\, z_3}/\mumu_{m'}(1,\, a',\, -1) 
\supset C_1 =(z_1\axis)/\mumu_{m'},
\end{equation*}
where $0 < a' < m'$ and $\gcd(a',\, m') =1$.
\end{lemma}

\begin{proof}
As in \cite[(9.6)]{Mori:flip}, we may deform a curve germ of $C_1$ in $X$
to a curve germ $(X_\lambda,\, C_{1,\lambda})$ of type~\typec{k2A} with only ordinary points
in such a way that the deformation is trivial outside
of some neighborhood of $R$. This trivially extends to a deformation $(X_\lambda,\, C_{\lambda})$
of $X$.
By \cite[Theorem~3.2]{MP:IA} every small deformation of $X$ is an extremal curve germ. 
Since we are going to derive a contradiction, it is enough to work on
$(X_\lambda,\, C_{\lambda})$.
\end{proof}

\begin{lemma}
\label{lemma:xic.3.3}
There exists a coordinate system
\begin{equation*}
(X,\, P)=\CC^3_{y_1,\, y_2,\, y_3}/\mumu_m(2,\, m-2,\, 1) \supset C_1=(y_2\axis)/\mumu_m
\end{equation*}
such that $C^{\sharp}_2$, the inverse image of $C_2$ by
the index-$1$ cover, is the locus of
$(t^2,\, t^{m-2},\, 0)$ ($t \in \CC$). Furthermore, we have
$m' \ge 3$ and $2(m'-a')<m'$.
\end{lemma}

\begin{proof}
If we ignore $C^{\sharp}_1$, this is in \cite[(A,3)]{Mori:flip}. Since $C_1$ is 
of type~\typec{k2A}, we see that $C^{\sharp}_1$ is smooth and there exist
$i,j \in \{1,\, 2,\, 3\}$ such that $y_i|_{C^{\sharp}_1}$ is a coordinate
for $C^{\sharp}_1$ and $\wt(y_i) + \wt(y_j) \equiv 0 \mod m$ \cite[(9.3)]{Mori:flip}.
Hence $i\in \{1,\, 2\}$. 

Assume that $i= 1$. 
We are going to compute the left hand side of \eqref{eq:ref(z:xic.3.6).2}.
We see $\upomega_{X^{\sharp}}=\Omega\cdot \OOO_{X^{\sharp}}$ with
$\Omega=\dd y_1 \wedge \dd y_2 \wedge \dd y_3$, and
$\III^{\sharp}_{C_2} = (y_3,\, y_1^{m-2}-y_2^2)$. Since $y_1|_{C_1^{\sharp}}$ is a
coordinate of $C_1^{\sharp}$ and $y_2$, $y_3$ are $\mumu_m$-semi-invariants,
we see
$y_2\cdot \OOO_{C_1^{\sharp}} \subset y_1^{m-1}\cdot \OOO_{C_1^{\sharp}}$ and
$y_3\cdot \OOO_{C_1^{\sharp}} \subset y_1^{(m+1)/2}\cdot \OOO_{C_1^{\sharp}}$.
Thus
\begin{equation*}
\III^{\sharp}_{C_2} \cdot \OOO_{C^{\sharp}_1} \subset 
\bigl(y_1^{(m+1)/2},\, y_1^{m-2}\bigr)\cdot \OOO_{C_1^{\sharp}} =
y_1^{(m+1)/2}\cdot \OOO_{C_1^{\sharp}}
\end{equation*}
because $m \ge 5$. Then $y_1^{(m-1)/2} \Omega$ induces a non-zero
$\mumu_m$-invariant element of 
$\upomega_{X^{\sharp}} \otimes \OOO_{C_1^{\sharp}}/\bigl(\III^{\sharp}_{C_2}\cdot \OOO_{C_1^{\sharp}}\bigr)$.
This contradicts \eqref{eq:ref(z:xic.3.6).2} and hence $i=2$. 

We choose a different coordinate system:
\begin{equation*}
\CC^3_{y'_1,\, y'_2,\, y'_3}/\mumu_m(2,\, m-2,\, 1) =
\CC^3_{y'_1,\, y'_2,\, y'_3}/\mumu_m\left(m-1,\, 1,\, \textstyle\frac{m-1}{2}\right) 
\supset C^{\sharp}_1 = (y'_2\axis)/\mumu_m
\end{equation*}
Then $C_2^{\sharp}$ is parametrized by 
$(t^2,\, t^{m-2}\cdot (\unit),\, t^{m+1} \cdot (\cdots))$.
A $\mumu_m$-equivariant coordinate change
\begin{equation*}
y''_1=y'_1, \quad
y''_2=y'_2 \cdot \phi_2(y'_1 y'_2), \quad
y''_3=y'_3+{y'_1}^{(m+1)/2}\phi_3(y'_1 y'_2)
\end{equation*}
with convergent power series $\phi_2(u),\phi_3(u)$ such that
$\phi_2(0) \neq0$ expresses $C_1^{\sharp}$ and $C_2^{\sharp}$ as
in Lemma~\ref{lemma:xic.3.3}.
Then by $0 > (K_X \cdot C_1)$, we see
\begin{equation*}
0 > -1 + \frac{m+1}{2m} + \frac{m'-a'}{m'} = 
\frac{m'-a'}{m'} - \frac{m-1}{2m} > \frac{m'-a'}{m'} - \frac{1}{2}
\end{equation*}
\cite[(2.3.1) and (4.9)]{Mori:flip}, whence $m' \ge 3$ and $2(m'-a')<m'$. 
This concludes the proof of Lemma~\ref{lemma:xic.3.3}.
\end{proof}

\subsection{}
\label{(z:xic.3.7)}
We fix the meaning of the symbols in Lemma~\ref{lemma:xic.3.3} during the proof of Theorem~\ref{thm:IC}.
In \cite[\S\S8--9]{Mori:flip} and \cite{KM92}, we used the symbol $P^{\sharp}$ to denote
a $\mumu_m$-invariant divisor on $C^{\sharp}$. Since $C$ is reducible in
our case, we use the symbol $P_i^{\sharp}$ to denote a divisor on $C_i^{\sharp}$
defined by a $\mumu_m$-semi-invariant coordinate, $i\in \{1,\, 2\}$ to avoid
confusion. Furthermore $\wt_i$ denotes the weight induced by $C_i$, that is
\begin{equation}
\label{eq:IC:def-wt1-2}
\begin{array}{lll}
\wt_1(y_1,\, y_2,\, y_3) &\equiv& \left(m-1,\, 1,\, \textstyle \frac{m-1}{2}\right), 
\\
\wt_2(y_1,\, y_2,\, y_3) &\equiv& (2,\, m-2,\, 1).
\end{array}
\end{equation} 
By \cite[(A.3.2)]{KM92}, there is an $\ell$-splitting
\begin{equation}
\label{eq:ref(z:xic.3.7).1}
\gr^1_{C_2}\OOO = (4P_2^{\sharp}) \toplus (-1+(m-1) P_2^{\sharp}),
\end{equation}
in which the factor $(4P_2^{\sharp})$ is unique. We note that
$\gr^1_{C_2} \OOO$ has an $\ell$-free $\ell$-basis 
$y_3$, $(y_1^{m-2}-y_2^2)$, and the $\ell$-invertible
sheaf $(4P_2^{\sharp})$ has an $\ell$-free $\ell$-basis $v$, where
\begin{equation}
\label{eq:refz:xic.3.7).2}
v = \lambda_1 y_1^{\frac{m-5}2}y_3 + \mu_1 \left(y_1^{m-2}-y_2^2\right)
\end{equation}
for some $\lambda_1$, $\mu_1 \in \OOO_{C_2}$.

Thus $\mu_1(P) \neq0$ if $m > 5$. To avoid the special case 
$\mu_1(P)=0$ when $m=5$, we use a deformation method \cite[(9.6)]{Mori:flip}.

\begin{lemma}[Deformation of patching]
\label{lemma:xic.3.8)}
We may assume $\mu_1(P)
\neq0$, that is $y_3$ and $v$ form an $\ell$-free $\ell$-basis of
$\gr^1_{C_2} \OOO$ at~$P$.
\end{lemma}

\begin{proof} 
By \cite[Theorem~3.2]{MP:IA} every small deformation of $X$ is an extremal curve germ.
So, to derive a contradiction with~\ref{setup:IC}, we take small neighborhoods $U_1 \Supset V_1$ of $C_1$ in $X$ and a small neighborhood $U_2$ of $C_2 \setminus V_1$ in
$X\setminus V_1$ and we deform the patching of $U_1$ and $U_2$ (cf.
\cite[(9.6)]{Mori:flip}) keeping
$U_1, V_1$, and $U_2$ unchanged. Then for a small deformation of $X$, we have
$\mu_1(P) \neq0$ as in \cite[(9.6.4)]{Mori:flip}. 
\end{proof}

\begin{lemma}
\label{lemma:xic.3.10)}
We have an $\ell$-splitting
$\gr^1_{C_1}\OOO= \LLL \toplus \MMM$
with $\ell$-isomorphisms
\begin{align*}
\LLL &\simeq (P_1^{\sharp} + R^{\sharp}), 
\\
\MMM &\simeq\gr^0_{C_1} \upomega \simeq
\bigl(-1+\textstyle\frac{m+1}2P_1^{\sharp}+(m'-a')R^{\sharp}\bigr).
\end{align*}
Furthermore, $\gr^1_{C_1}\OOO$ \textup(resp. $\LLL$\textup) has an $\ell$-free
$\ell$-basis
$y_1,\, y_3$ \textup(resp. $y_1+\nu y_3 y_2^{\frac{m-1}2}$ for some $\nu\in \OOO_{C_1,P}$\textup) at~$P$.
\end{lemma}

\begin{proof}
We note $m,\, m' \ge 3$ by Lemma~\ref{lemma:xic.3.3} and Setup~\ref{setup:IC}. So we
need to quote from \cite[\S9]{Mori:flip}, where $P'$ is our $R$. The
$\ell$-splitting is given in \cite[(9.9)]{Mori:flip} and $\MMM$ is given in \cite[(9.9.1)]{Mori:flip}. We
have $\LLL \simeq \bigl(q(P)P_1^{\sharp} + q(R)R^{\sharp}\bigr)$,
where $q(P):=\qldeg(\LLL,P)$ and $q(R):=\qldeg(\LLL,R)$ \cite[(9.5.1)]{Mori:flip}.
Further, $q(P)=q(R)=1$ by \cite[(9.8)]{Mori:flip} (since $P$ and $R$ are ordinary, $\ell(P) = \ell(R)=0$, see \eqref{eq:ell(P)}).
\end{proof}

\begin{lemma}
\label{lemma:xic.3.12)}
$\gr^1_C\OOO\totimes \OOO_{C_1}$ has an
$\ell$-splitting
$\gr^1_C \OOO\totimes \OOO_{C_1} = \AAA_1 \toplus \BBB_1$, where
\begin{equation*}
\AAA_1=\bigl(-1+(m-1)P_1^{\sharp} + R^{\sharp}\bigr), \quad \BBB_1=\bigl(-1+\textstyle\frac{m+1}2 P_1^{\sharp} +
(m'-a')R^{\sharp}\bigr)
\end{equation*}
such that the $\ell$-homomorphism 
\begin{equation*}
\gr^1_C \OOO \totimes \OOO_{C_1} \xlongrightarrow{\hspace*{17pt}} \gr^1_{C_1} \OOO = \LLL \toplus \MMM
\end{equation*}
\textup(see Lemma~\xref{lemma:xic.3.10)}\textup) is the direct sum of $\AAA_1 \to \LLL$ and $\BBB_1 \to
\MMM$. Furthermore, $\gr^1_C \OOO \totimes \OOO_{C_1}$ has an $\ell$-free
$\ell$-basis $y_3,\, y_1(y_1^{m-2}-y_2^{2})$ at~$P$, and $\AAA_1$ has an $\ell$-free $\ell$-basis with
$\wt_2 \equiv \wt_2(y_2)$.
\end{lemma}

\begin{proof}
We see that
\begin{equation*}
\coker \bigr[\III_C^{\sharp} \xlongrightarrow{\hspace*{17pt}} \bigl(\III_{C_1}^{\sharp}/{\III_{C_1}^{\sharp 2}}\bigr)\bigr]
= \III_{C_1}^{\sharp}/\bigl(\III_C^{\sharp}+{\III_{C_1}^{\sharp 2}}\bigr)
\\
=(y_1,\, y_3)/\bigl(y_3,\, y_1^2,\, y_1y_2^2\bigr)
\end{equation*}
has $\mumu_m$-invariant part $\CC \cdot y_1y_2$. This is the cokernel of 
$\gr^1_C \OOO \totimes \OOO_{C_1} \to\gr^1_{C_1} \OOO$ as a sheaf homomorphism.
Thus there is an isomorphism $\gr^1_C \OOO \totimes \OOO_{C_1} \simeq \OOO(-1)^{\oplus 2}$ as sheaves by
$\gr^1_{C_1} \OOO \simeq \OOO_{C_1} \oplus \OOO_{C_1}(-1)$ and 
$H^1(\gr^1_C \OOO \otimes \OOO_{C_1}) = 0$ \cite[Proof of (2.3)]{Mori:flip}, \cite[Lemma~3.5.1]{MP-1p}. 
In view of the $\ell$-free $\ell$-basis, we see
the $\ell$-splitting \cite[(2.8)]{KM92}:
\begin{equation*}
\gr^1_C \OOO \totimes \OOO_{C_1} =
\begin{cases}
(-1+(m-1)P_1^{\sharp} + R^{\sharp}) \toplus (-1 + \frac{m+1}2P_1^{\sharp} + (m'-a')R^{\sharp}), \text{or} 
\\
(-1+\frac{m+1}2 P_1^{\sharp} + R^{\sharp}) \toplus (-1 + (m-1) P_1^{\sharp} + (m'-a')R^{\sharp}).
\end{cases}
\end{equation*}
If $m'-a'=1$, then these two cases coincide. We claim that the second case is impossible
if $m'-a' > 1$. Indeed 
\begin{equation*}
\FFF:=\bigl(-1+(m-1)P_1^{\sharp}+(m'-a')R^{\sharp}\bigr)
\end{equation*} 
in the second case has no non-zero
$\ell$-homomorphism to $\LLL$ and $\MMM$ by
\begin{equation*}
\FFF^{\totimes(-1)} \totimes \LLL =(-1+2P_1^{\sharp}+(a'+1)R^{\sharp}),\quad
\FFF^{\totimes(-1)} \totimes \MMM =\bigl(-1+\textstyle\frac{m+3}{2}P_1^{\sharp}\bigr),
\end{equation*}
$m \ge 5$ and $m'-a' \ge 2$. Thus we are in the first case. 
By a similar reason,
$\gr^1_C \OOO \totimes \OOO_{C_1} \to\gr^1_{C_1} \OOO$ is a direct sum as claimed if $m'-a'>1$. If $m'-a'=1$,
we can rearrange the $\ell$-splitting $\gr^1_C \OOO \totimes \OOO_{C_1} =\AAA_1\toplus \BBB_1$ by a non-zero homomorphism 
$ \BBB_1\to \AAA_1$ to attain the direct sum.
For the $\ell$-free
$\ell$-basis, we just note 
\begin{equation*}
\wt_2 \Bigl(y_1(y^2_2 - y_1^{m-2}) + \nu y_3 y_2^{\frac{m+3}2}\Bigr) \equiv \wt_2(y_2)
\qedhere
\end{equation*}
\end{proof}

\begin{lemma}
\label{lemma:xic.3.14)}
$\gr^1_C \OOO \totimes \OOO_{C_2}$ fits in an
$\ell$-exact sequence
\begin{equation*}
0 \xlongrightarrow{\hspace*{17pt}} \AAA_2\xlongrightarrow{\hspace*{17pt}}\gr^1_C \OOO \totimes \OOO_{C_2} \xlongrightarrow{\hspace*{17pt}} \BBB_2 \xlongrightarrow{\hspace*{17pt}} 0,
\end{equation*}
where $\AAA_2 = \bigl(2P_2^{\sharp}\bigr)$,\, $\BBB_2 = \bigl(-1+(m-1)P_2^{\sharp}\bigr)$, and $\AAA_2$ has an $\ell$-free
$\ell$-basis with $\wt_2 \equiv \wt_2(y_2)$ at~$P$.
\end{lemma}

\begin{proof}
In 
$(\gr^1_{C_2}\OOO)^{\sharp} = \OOO_{C_2}^{\sharp} \cdot v \oplus \OOO_{C_2}^{\sharp}\cdot y_3$, 
with $v$ as in \eqref{eq:refz:xic.3.7).2} we see
\begin{equation*}
(\gr^1_C\OOO\totimes \OOO_{C_2})^{\sharp} = \OOO_{C_2}^{\sharp}\cdot y_1 (y_2^2 - y_1^{m-2})+ \OOO_{C_2}^{\sharp}\cdot y_3 
= y_1 \OOO_{C_2}^{\sharp} \cdot v + \OOO_{C_2}^{\sharp}\cdot y_3
\end{equation*}
by Lemma~\ref{lemma:xic.3.8)}. Since $v$ is an $\ell$-free $\ell$-basis of
$(4P^{\sharp})$, if we set
$y_1(4P^{\sharp})$ as $\AAA_2$ then we see $\wt_2(y_1 v) \equiv \wt_2(y_2)$ and
the statement of~\ref{lemma:xic.3.14)} follows from \eqref{eq:ref(z:xic.3.7).1}. 
\end{proof}

\begin{lemma}
\label{lemma:xic.3.16)}
The $\ell$-subbundles
\begin{equation*}
\begin{array}{llll}
\AAA_1 & \subset&\gr^1_C \OOO \totimes \OOO_{C_1} &(\text{cf. Lemma~\xref{lemma:xic.3.12)}}),
\\[7pt]
\AAA_2 & \subset&\gr^1_C \OOO \totimes \OOO_{C_2} &(\text{cf. Lemma~\xref{lemma:xic.3.14)}})
\end{array}
\end{equation*}
are restrictions of an $\ell$-subbundle $\AAA \subset\gr^1_C \OOO$. In
particular, if we set the $\ell$-invertible sheaf $\BBB=\gr^1_C \OOO /\AAA$, then
we have 
$\ell$-isomorphisms $\BBB \totimes \OOO_{C_i} = \BBB_i$ for $i\in \{1,\, 2\}$.
\end{lemma}

\begin{proof}
We note an exact sequence
\begin{equation}
\label{eq:new:es}
0 \xlongrightarrow{\hspace*{17pt}} \OOO_{C^{\sharp}} \xlongrightarrow{\hspace*{17pt}} \OOO_{C_1^{\sharp}} \oplus \OOO_{C_2^{\sharp}} \xlongrightarrow{\hspace*{17pt}} 
\OOO_{C_1^{\sharp} \cap C_2^{\sharp}} \xlongrightarrow{\hspace*{17pt}} 0
\end{equation}
and an isomorphism
\begin{equation}
\label{eq:new:cap}
\OOO_{C_1^{\sharp} \cap C_2^{\sharp}} \simeq \CC[y_1,\, y_2,\, y_3]/(y_1,\, y_2^2,\, y_3)
\simeq \CC[y_2]/(y_2^2).
\end{equation}
We note that $\gr^1_C \OOO$ is $\ell$-free at~$P$ and has an $\ell$-free $\ell$-basis
$\phi_1 = y_3$ and $\phi_2 = y_1(y_2^2 - y_1^{m-2})$, and that $\AAA_1$ and
$\AAA_2$ have
$\ell$-free $\ell$-bases both with $\wt_2 \equiv \wt_2(y_2) \equiv \wt_2(\phi_2)$
by Lemmas~\ref{lemma:xic.3.12)} and~\ref{lemma:xic.3.14)}.
In the exact sequence
\begin{equation*}
0 \to\gr^1_C \OOO^{\sharp} 
\xlongrightarrow{\hspace*{7pt}\alpha\hspace*{7pt}}
{\gr^1_C\OOO^{\sharp}} \otimes (\OOO_{C_1}^{\sharp} \oplus \OOO_{C_2}^{\sharp}) 
\xlongrightarrow{\hspace*{17pt}} 
\CC[y_2]/(y_2^2) \phi_1 \oplus \CC[y_2]/(y_2^2) \phi_2 \to 0
\end{equation*}
obtained from the above, we consider the part of $\wt_2 \equiv \wt_2(\phi_2)$.
Then
\begin{equation*}
\coker (\alpha)_{\wt_2 \equiv \wt_2(\phi_2)} = \CC \cdot \phi_2
\end{equation*}
because $\wt_2(y_2) \not\equiv 0$ and $\wt_2 (\phi_2) \not\equiv \wt_2 (y_2 \phi_1)$.
Therefore the $\ell$-free $\ell$-bases of $\AAA_1$ and $\AAA_2$ (up to constant
multiplication) lift to a section of $\gr^1_C \OOO^{\sharp}$. 
\end{proof}
By $\ell$-invertible sheaves $\AAA$ and $\BBB$ of Lemma~\ref{lemma:xic.3.16)}, we have $\mathrm{E}(\JJJ, 2)$ of Proposition~\ref{thicken-I}, and Theorem~\ref{M88-8.12} with $d=2$ implies
\begin{multline*}
\ldeg(\AAA) + 2 \ldeg(\BBB)
= \left(-1 + \textstyle\frac{m-1}{m} + \frac1{m'}\right) + \textstyle\frac2m +\\
+ 2\left(\left(-1 + \textstyle\frac{m+1}{2m} + \textstyle\frac{m'-a'}{m'}\right)+ \left(-1+\textstyle \frac{m-1}m\right)\right) 
=\textstyle \frac{m'+1-2a'}{m'} \ge 0.
\end{multline*}
On the other hand, we have
\begin{equation*}
0 > (K_X \cdot C_1) = -1 +\textstyle \frac{m+1}{2m} + \frac{m'-a'}{m'} = \textstyle\frac{m+1}{2m} -\frac{a'}{m'}.
\end{equation*}
Thus we have
\begin{equation*}
\textstyle\frac{m'+1}2 \,\ge \, a' > \textstyle\frac{m+1}{2m} \cdot m' > \frac{m'}2.
\end{equation*}
Whence $2a'=m'+1$, $m>m'$, $\ldeg(\AAA) + 2 \ldeg(\BBB)=0$, and $(X,\, C)$ is a $\QQ$-conic bundle
by Theorem~\ref{M88-8.12}.
Under the notation of Proposition~\ref{thicken-J} with $d=2$, we have an $\ell$-exact sequence
\begin{equation*}
0 \xlongrightarrow{\hspace*{17pt}}\BBB^{\totimes 2} \xlongrightarrow{\hspace*{17pt}} \gr^2(\OOO,\, \JJJ) \xlongrightarrow{\hspace*{17pt}} \AAA \xlongrightarrow{\hspace*{17pt}} 0.
\leqno{\mathrm{E}(\JJJ)}
\end{equation*}
The extension class $[E(\JJJ)]$ belongs to $H^1(\AAA^{\totimes -1} \totimes \BBB^{\totimes 2})$.
We see 
\begin{eqnarray*}
\AAA^{\totimes (-1)} \totimes \BBB^{\totimes 2}\totimes\OOO_{C_1} &= 
&-\bigl(-1+(m-1)P_1^{\sharp}+R^{\sharp}\bigr)+
2\bigl(-1+\textstyle\frac{m+1}{2}P_1^{\sharp} +\textstyle\frac{m'-1}{2}R^{\sharp}\bigr)
\\
&=&-1+2 P_1^{\sharp}+(m'-2)R^{\sharp},
\end{eqnarray*}
and
\begin{eqnarray*}
\AAA^{\totimes (-1)} \totimes \BBB^{\totimes 2}\totimes\OOO_{C_2} 
&=&
-\bigl(2 P_2^{\sharp})+2(-1+(m-1)P_2^{\sharp}\bigr)
\\
&=&-1+(m-4)P_2^{\sharp}.
\end{eqnarray*}
Hence $H^1(C,\, \AAA^{\totimes (-1)} \totimes \BBB^{\totimes 2})=0$ by
\[
\left((\AAA^\sharp)^{\totimes (-1)} \totimes (\BBB^\sharp)^{\totimes 2}\totimes \OOO_{C_1^\sharp \cap C_2^\sharp }\right)^{\mumu_m}=0
\]
(cf. \eqref{eq:new:es} and~\eqref{eq:new:cap}),
and $\JJJ$ extends to $\KKK$ which is $(1,3)$-monomializable. Hence by Theorem~\ref{M88-8.12} with $d=3$, we have 
\begin{equation*}
\ldeg(\AAA) + 3 \ldeg(\BBB)= \ldeg(\BBB)=-\textstyle\frac{m+m'}{2mm'}\ge 0,
\end{equation*}
which is a contradiction.
Thus a \typec{k2A} component cannot meet a \typec{IC} component. This completes the proof of Theorem~\ref{thm:IC}. 

\section{Cases \typec{k2A}+\typec{kAD} and \typec{k2A}+\typec{k3A}.}
\label{sect:kAD-k3A}

The purpose of this section is to disprove the following situation.

\begin{setup}
\label{notation:xkad.1)}
Let $(X,\, C)$ be an extremal curve germ such that
$C=C_1 \cup C_2$, where $C_1$ and $C_2$ are irreducible, $C_1$ is of type
\typec{k2A}, and $C_2$ is of type~\typec{kAD} or \typec{k3A}.

In this situation $X$ has exactly three non-Gorenstein points: 
\begin{itemize}
\item 
an ordinary point $P$ of odd index $m \ge 3$, where $\{P\}=C_1 \cap C_2$,
\item 
$Q$ of index $m' \ge 3$ on $C_1$,
\item 
$R$ of index $2$ on $C_2$.
\end{itemize}
Moreover, if $(X,C_2)$ is of type~\typec{k3A}, then $C_2$ has a type~\typec{III} point $S$ of $X$.

\begin{scase}
Indeed, since the germ $(X,C_2)$ 
is of type~\typec{kAD} or \typec{k3A}, it has exactly two non-Gorenstein points: an ordinary point $P$ 
of odd index $m \ge 3$ and an index $2$ point $R$ (see \cite[2.13, 2.12]{KM92}, and \cite{Mori:err}). By \cite[Corollary~1.4(iv)]{MP20} \ 
$C_1$ does not contain index $2$ points, so $C_1 \cap C_2=\{P\}$.
Moreover, in the case \typec{k3A} there is a type~\typec{III} point on $C_2$ and $R$ is an ordinary point \cite[2.12]{KM92}.
\end{scase}
\end{setup} 

\begin{sremark}
\label{rem:Cl}
Note that under the assumption of \xref{notation:xkad.1)} the Weil divisor class group $\Cl(X)$ is torsion free (cf.~\ref{not:app}),
i.e. $(X,C)$ is primitive.
If furthermore $(X,C)$ is a $\QQ$-conic bundle germ, then the base $(X,o)$ must be smooth.

Indeed, since $(X,C_2)$ is primitive, the group $\Cl(X)_{\tors}$ 
is contained in the kernel of the natural homomorphism $\Cl(X) \to \Cl(X,C_2)$.
Hence $\Cl(X)_{\tors}\subset \Cl(X,C_1)_{\tors}$ and the image of $\Cl(X)_{\tors}$ in $\Cl(X,P)$
is trivial by Corollary~\ref{cor:Clsc}. On the other hand, since $(X,C_1)$ is locally primitive at~$P$ and $Q$,
we have an injection $\Cl(X,C_1)_{\tors}\subset \Cl(X,P)$ again by Corollary~\ref{cor:Clsc},
a contradiction.
\end{sremark}

\begin{theorem}\label{thm:KAD-k3A}
The case \xref{notation:xkad.1)} does not occur.
\end{theorem} 

The proof of Theorem~\ref{thm:KAD-k3A} will be done in a series
of lemmas.

\begin{slemma}[Deformation at $Q$]
\label{lemma:xkad.3.1)}
We may assume $Q$ is ordinary, that is,
$(X,\, Q)$ is a cyclic quotient singularity. In particular by \cite[(9.3)]{Mori:flip}, there
is a coordinate system:
\begin{equation*}
(X,\, Q)=\CC^3_{z_1,\, z_2,\, z_3}/\mumu_{m'}(1,a',-1) \supset C_1=(z_1\axis)/\mumu_{m'},
\end{equation*}
where $0 < a' < m'$ and $\gcd(a',m')=1$.
\end{slemma}
The proof of Lemma~\ref{lemma:xkad.3.1)} is exactly the same as one of Lemma~\ref{lemma:xic.3.1}.

\begin{scase}
\label{case:xkad.3.2z)}
By deformation methods as in \cite[6.1]{MP20}, we can successively deform $(X,C_2)$ as follows:
\par\medskip\noindent\qquad
\begin{tabularx}{0.7\textwidth}{lcl}
\typec{kAD} with $\ell(R)=2$ & $\xRightarrow{\text{\text{\cite[Lemma~5.5]{MP20}}}}$ & \typec{k3A};
\\
\typec{k3A} with $m \ge 5$ & $\xRightarrow{\text{\text{\cite[Lemma~5.3]{MP20}}}}$ & \typec{kAD} with $\ell(R)=0$;
\\
\typec{kAD} with $\ell(R) = 1$& $\xRightarrow{\text{\cite[Lemma~5.7]{MP20}}}$ & \typec{kAD} with $\ell(R)=0$.
\end{tabularx}
\par\medskip\noindent
Indeed, \cite[Remark~1]{Mori:err} explains that \typec{kAD} in (2.13) of \cite{KM92} comes out of two sources; one is (2.13.3.1) via
(2.13.10) and (2.13.12) where $\ell(R) \le 1$ and $m \ge 5$; and the other is (2.13.3.2) via (2.13.4) where $\ell(R)=2$ and $m \ge 3$. Thus
\typec{kAD} with $m=3$ has $\ell(R)=2$ \cite[(2.13.3.2)]{KM92} and deforms to \typec{k3A} in the above deformations. Hence we arrive at
the situation where $R$ is a cyclic quotient singularity and
$C_2$ is of type~\typec{kAD} with $m \ge 5$ and $\ell(R)=0$ or \typec{k3A} with $m=3$. 
\end{scase}

\begin{slemma}
\label{lemma:xkad.3.3)a}
There exists a coordinate system
\begin{equation*}
(X,\, P)=\CC^3_{y_1,\, y_2,\, y_3}/\mumu_m\textstyle \bigl(1, \textstyle\frac{m+1}2,-1\bigr) \supset
\begin{cases}
C_2 = (y_1\axis)/\mumu_m,
\\
C_1 = (y_3\axis)/\mumu_m,
\end{cases}
\end{equation*}
where $m$ is odd and $m \ge 5$ \textup(resp. $m = 3$\textup)
if $C_2$ is of type~\typec{kAD} \textup(resp. \typec{k3A}\textup).
Furthermore, we have $m' \ge 3$ and $m'-a' < m'/2$.
\end{slemma}

The proof is very similar to that of Lemma~\ref{lemma:xic.3.3}.
If we ignore $C_1$, this is in \cite[(2.12.2)]{KM92} in case of \typec{k3A}. 
In case of \typec{kAD}, it follows from (2.13.3), 
(2.13.4), (2.13.9) and (2.13.10) of \cite{KM92}.

\begin{proof}
Since $C_1$ is of type~\typec{k2A}, we see that
$C_1^\sharp$ (the inverse image of $C_1$ by the index one cover $X^\sharp \to X$) is smooth
and there exist $i, j \in \{1,2,3\}$ such that $y_i |_{C_1^\sharp}$
is a coordinate for $C_1^\sharp$ and $\wt(y_i) + \wt(y_j) \equiv 0 \mod m$
\cite[(9.3)]{Mori:flip}. 

Assume that $i=1$. As in the proof of $i\neq 1$ in Lemma~\ref{lemma:xic.3.3}, we have
\begin{equation}
\label{eq:n-xkad.3.6)).1}
\bigl(\upomega_{X^\sharp} \otimes (\OOO/\III^\sharp_{C_1}+\III^\sharp_{C_2})\bigr)^{\mumu_m}=0. 
\end{equation} 

We see $\upomega_{X^\sharp} = \Omega \cdot\OOO_{X^\sharp}$ with 
$\Omega=\dd y_1 \wedge \dd y_2 \wedge \dd y_3$, and $\III^\sharp_{C_2}=(y_2,\, y_3)$.
Since $y_1 |_{C^\sharp_1}$ is a coordinate and $y_2$, $y_3$ are 
$\mumu_m$-semi-invariants, we see 
$y_2\cdot \OOO_{C_1^\sharp} \subset y_1^{\frac{m+1}2} \cdot \OOO^\sharp_{C_1}$ and
$y_3 \cdot \OOO^\sharp_{C_1} \subset y_1^{m-1}\cdot \OOO^\sharp_{C_1}$. Thus
\begin{equation*}
\III^\sharp_{C_2} \cdot \OOO^\sharp_{C_1} \subset y_1^{\frac{m+1}2} \cdot \OOO^\sharp_{C_1},
\end{equation*}
and $y_1^{\frac{m-1}2} \Omega$ induces a non-zero $\mumu_m$-invariant section
of $\upomega_{X^\sharp} \otimes \bigl(\OOO^\sharp_{C_1} /(\III^\sharp_{C_2} \cdot \OOO^\sharp_{C_1})\bigr)$. 
This contradicts \eqref{eq:n-xkad.3.6)).1}. Thus $i\neq 1$. 

If $m \ge 5$, then we have $i=3$. If $m=3$, then we can still assume
$i=3$ switching $y_2$ and $y_3$ if necessary. Thus $i=3$ in any case.
Since $y_3 |_{C_1^\sharp}$ is a coordinate for $C_1^\sharp$, we can make a 
$\mumu_m$-equivariant coordinate change
\begin{equation*}
y'_1=y_1-y_3^{m-1} \phi_1(y_3^m),\quad y'_2 = y_2 - y_3^{\frac{m-1}2} \phi_2(y_3^m), \quad g'_3 = y_3
\end{equation*}
with convergent power series $\phi_1(u), \phi_2(u)$ so that $C_1, C_2$ are
as described in Lemma~\ref{lemma:xkad.3.3)a}. 

Then by $0 > (K_X \cdot C_1)$, we have
\begin{equation*}
0 > -1 + \frac{m+1}{2m} + \frac{m'-a'}{m'} = \frac{m'-a'}{m'} - \frac{m-1}{2m} 
> \frac{m'-a'}{m'} - \frac12
\end{equation*}
\cite[(2.3.1) and (4.9)]{Mori:flip}, whence $m' \ge 3$ and $m'-a' < m'/2$. 
This completes the proof of Lemma~\ref{lemma:xkad.3.3)a}.
\end{proof} 

\subsection{}\label{case:xkad.3.7)g}
We fix the meaning of the symbols in Lemma~\ref{lemma:xkad.3.3)a} during the
proof of Theorem~\ref{thm:KAD-k3A}. As in Section~\ref{sect:IC}, we use the symbol $P^\sharp_i$ to denote a
divisor on
$C^\sharp_i$ defined by a $\mumu_m$-semi-invariant coordinate, $i\in \{1,\, 2\}$ to avoid
confusion. 
As in \eqref{eq:IC:def-wt1-2}, we use $\wt_i$ to denote the weight induced by $C_i$,
that is
\begin{equation}
\label{eq:kAD:def-wt1-2}
\begin{array}{lll}
\wt_1(y_1,\, y_2,\, y_3) &\equiv \bigl(-1,\, \textstyle \frac{m-1}2,\,1\bigr), 
\\
\wt_2(y_1,\, y_2,\, y_3) &\equiv \bigl(1,\, \textstyle\frac{m+1}2,\,-1\bigr). 
\end{array}
\end{equation}
In view of~\ref{case:xkad.3.2z)},
we treat \typec{kAD} and \typec{k3A} 
in two subcases: 
\begin{subcase}
\label{case:xkad.3.7)gcase-1}
The germ $(X,C_2)$ is of type~\typec{kAD} with $m\ge 5$ and $\ell(R)=0$ \cite[5.7]{MP20}, and
\end{subcase}
\begin{subcase}
\label{case:xkad.3.7)gcase-4}
The germ $(X,C_2)$ is of type~\typec{k3A} with $m=3$ \cite[5.4]{MP20}.
\end{subcase} 

Then we have a coordinate system at $R$
\begin{equation*}
(X,\, R)=\CC^3_{v_1,\, v_2,\, v_3}/\mumu_2(1,1,1) \supset C_2 =(v_1\axis)/\mumu_2,
\end{equation*}
Furthermore, we have an $\ell$-splitting (see \eqref{eq:ref(z:xic.3.7).1})
\begin{equation}
\label{gr1_C-O}
\gr^1_{C_2}\OOO =
\begin{cases}
\bigl(\frac{m-1}2 P_2^\sharp + R^\sharp\bigr) \toplus \bigl(-1 + P_2^\sharp + R^\sharp\bigr) &
\text{$m \ge 5$, \quad case~\ref{case:xkad.3.7)gcase-1}}
\\
\bigl(-1+ P_2^\sharp+R^\sharp\bigr) \toplus \bigl(-1 + P_2^\sharp + R^\sharp\bigr) &
\text{$m=3$, \quad case~\ref{case:xkad.3.7)gcase-4}}
\end{cases}
\end{equation}
Indeed, the expression for $\gr^1_{C_2}\OOO$ in the case~\ref{case:xkad.3.7)gcase-1}
follows from \cite[(5.26)]{MP20}.
In the case~\ref{case:xkad.3.7)gcase-4} we have 
\[
\gr_{C_2}^0\upomega ^*=(-1+2P_2^\sharp+R^\sharp)
\]
by \cite[(5.15)]{MP20}, hence
\[
\gr_{C_2}^0\upomega=(-1+P_2^\sharp+R^\sharp).
\]
because $P$ and $R$ are points of indices $3$ and $2$, respectively. Further,
\[
\gr_{C_2}^1\upomega ^*=(0)\toplus (0)  
\]
by \cite[(5.17)]{MP20}. Since $\upomega_X$ is $\ell$-invertible on $X$, we have
\[
\gr_{C_2}^1 \OOO= \gr_{C_2}^1\upomega ^*\otimes \upomega_X=
\gr_{C_2}^1\upomega ^*\otimes  \gr_{C_2}^0\upomega=\bigl(-1+ P_2^\sharp+R^\sharp\bigr) \toplus \bigl(-1 + P_2^\sharp + R^\sharp\bigr),
\]
which is \eqref{gr1_C-O} in the case~\ref{case:xkad.3.7)gcase-4}.
We note that the axial multiplicity of $X$ at $R$ is equal to~$1$.

\begin{slemma}
\label{lemma:xkad.3.10)}
For $\gr^1_{C_1} \OOO$, we have an $\ell$-splitting
$\gr^1_{C_1} \OOO = \LLL \toplus \MMM$ with $\ell$-isomorphisms
\begin{align*}
\LLL &\simeq \bigl(P_1^\sharp + Q^\sharp\bigr), 
\\
\MMM &\simeq \gr^0_{C_1} \upomega \simeq \bigl(-1 + \textstyle\frac{m+1}2 P_1^\sharp + (m'-a') Q^\sharp\bigr).
\end{align*}
Furthermore,
$\gr^1_{C_1} \OOO$ \textup(resp. $\LLL$\textup) has an $\ell$-free $\ell$-basis $y_1, y_2$
\textup(resp. $y_1+\nu y_2 y_3^{\frac{m-3}2}$ for some $\nu \in \OOO_{C_1,P}$\textup)
at~$P$.
\end{slemma}
This is exactly the same as \cite[(2.13.8-9)]{KM92} and Lemma~\xref{lemma:xic.3.10)}
except that the coordinates are numbered
differently.

\begin{slemma}
\label{lemma:xkad.3.12)q}
$\gr^1_C \OOO \totimes \OOO_{C_1}$ has an $\ell$-splitting 
$\gr^1_C\OOO \totimes \OOO_{C_1} = \AAA_1 \toplus \BBB_1$, where
\begin{equation*}
\AAA_1 = \bigl(-1 + \textstyle\frac{m+1}2 P_1^\sharp +(m'-a')Q^\sharp\bigr), \qquad \BBB_1 = (Q^\sharp).
\end{equation*}
Furthermore, the defining ideal $\III_C^\sharp$ (at~$P^\sharp$) of $C$
is generated by $y_1y_3$ and $y_2$.
\end{slemma}

\begin{proof}
\label{case:xkad.3.13)w} 
This is similar to Lemma~\ref{lemma:xic.3.12)}.
At $P$, we see that
\begin{align*}
\coker \left(\III_C^\sharp \xlongrightarrow{\hspace*{17pt}} \bigl(\III^\sharp_{C_1}/\III^{\sharp 2}_{C_1}\bigr)\right) &
= \III^\sharp_{C_1}/\bigl(\III^\sharp_C + \III^{\sharp 2}_{C_1}\bigr) 
\\
&= (y_1,\, y_2)/\bigl(y_2,\, y_1y_3,\, y_1^2\bigr) 
\\
&= \CC \cdot y_1.
\end{align*}
Since
\begin{equation*}
\gr^1_{C_1} \OOO =\bigl(P^\sharp_1 + Q^\sharp\bigr) 
\toplus \bigl(-1 + \textstyle\frac{m+1}2 P^\sharp_1 + (m'-a')Q^\sharp\bigr),
\end{equation*}
we obtain~\ref{lemma:xkad.3.12)q} by the $\ell$-injection $\gr^1_C\OOO \totimes \OOO_{C_1} \to
\gr^1_{C_1}
\OOO$.
\end{proof} 

\begin{slemma}\label{lemma:xkad.3.14)o}
$\gr^1_C \OOO \totimes \OOO_{C_2}$ has an $\ell$-splitting
$\gr^1_C\OOO\totimes \OOO_{C_2} = \AAA_2 \toplus \BBB_2$, where
\begin{equation*}
\AAA_2 = \bigl(i+\textstyle\frac{m-1}2 P^\sharp_2 +R^\sharp\bigr),\qquad \BBB_2=(-1+R^\sharp),
\end{equation*}
where $i=0$ \textup(resp. $-1$\textup) for \xref{case:xkad.3.7)gcase-1}
\textup(resp. for \xref{case:xkad.3.7)gcase-4}\textup).
Furthermore, the $\ell$-injection 
\[
\gr^1_C\OOO \totimes \OOO_{C_2} \longrightarrow \gr^1_{C_2} \OOO
\]
induces an $\ell$-isomorphism of $\AAA_2$ with the first summand of the $\ell$-decomposition \eqref{gr1_C-O},
where for $m=3$ we need to choose the $\ell$-decomposition \eqref{gr1_C-O} properly.
\end{slemma}

\begin{proof}
This is similar to Lemma~\ref{lemma:xic.3.16)}.
Similarly to the proof of Lemma~\ref{lemma:xkad.3.12)q}, we have
\begin{align*}
\coker \Bigl(\III^\sharp_C \xlongrightarrow{\hspace*{17pt}} \bigl(\III^\sharp_{C_2}/\III^{\sharp 2}_{C_2}\bigr)\Bigr) &
= \III^\sharp_{C_2}/\bigl(\III^\sharp_C + \III^{\sharp 2}_{C_2}\bigr)
\\
&=(y_2,\, y_3)/\bigl(y_2,\, y_1y_3,\, y_3^2\bigr) 
\\
&= \CC \cdot y_3.
\end{align*}
Thus the $\ell$-injection $\gr^1_C \OOO \totimes \OOO_{C_2} \to \gr^1_{C_2} \OOO$ 
implies~\ref{lemma:xkad.3.14)o}
by the $\ell$-splitting
\begin{equation*}
\gr^1_{C_2} \OOO = \bigl(i+\textstyle\frac{m-1}2 P^\sharp_2 + R^\sharp\bigr) \toplus (-1+P_2^\sharp + R^\sharp)
\end{equation*}
in \eqref{gr1_C-O}. 
Indeed $\CC\cdot y_3$ has an $\ell$-surjective map only from the second factor
of $\gr_{C_2}^1 \OOO$ if
$m \ge 5$. If $m=3$, then we are in~\ref{case:xkad.3.7)gcase-4};
the summands of $\gr_{C_2}^1 \OOO_X$ are isomorphic to each other and the first assertion is proved.
The rest is obvious. 
\end{proof} 

\begin{slemma}
\label{lemma:xkad.3.16)}
Under the notation of Lemma~\xref{lemma:xkad.3.3)a}, let $\EEE^\sharp$ be a free $\OOO_{C^\sharp}$-module at~$P$ with
$\mumu_m$-action. Let $\FFF^\sharp_i$ be a free $\OOO^\sharp_{C_i}$-submodule of $\EEE^\sharp_i =
\EEE^\sharp \otimes \OOO_{C^\sharp_{C_i}}$ such that $\mumu_m(\FFF^\sharp_i) = \FFF^\sharp_i$ and
$\EEE^\sharp_i/\FFF^\sharp_i$ is free for $i\in \{1,\, 2\}$. If we have 
$\FFF^\sharp_1 \otimes \CC_{P^\sharp} =\FFF^\sharp_2 \otimes \CC_{P^\sharp}$ in 
$\EEE^\sharp \otimes \CC_{P^\sharp}$, then there is a free $\OOO_{C^\sharp}$-submodule
$\FFF^\sharp$ of $\EEE^\sharp$ such that $\FFF^\sharp \otimes \OOO^\sharp_{C_i} = \FFF^\sharp_i$ and
$\EEE^\sharp/\FFF^\sharp$ is free.

In particular, assume that $\EEE^\sharp$ has an $\ell$-free $\ell$-basis whose
$C^\sharp_1$-wts are mutually distinct. Then the associated $\ell$-free $\OOO_C$-module
$\EEE$ at~$P$ has the property: arbitrary $\ell$-splittings 
$\EEE\totimes \OOO_{C_j} = \FFF^1_j \toplus \cdots \toplus \FFF^r_j$, $j=1,\, 2$ into 
$\ell$-invertible sheaves such that
\begin{equation*}
\sum^2_{j=1}\wt_j\bigl(\text{$\ell$-free $\ell$--basis of $\FFF^i_j$}\bigr) \equiv 0 \mod m
\end{equation*}
for all $i$ induce an $\ell$-splitting $\EEE=\FFF^1 \toplus \cdots \toplus \FFF^r$ such
that $\FFF^i \totimes \OOO_{C_j} = \FFF^i_j$ for all $i,j$.
\end{slemma}

\begin{proof}
We note by Lemma~\ref{lemma:xkad.3.3)a} that $\OOO_{C^\sharp_1 \cap
C^\sharp_2} =
\CC$, that is $\III^\sharp_{C_1}+\III^\sharp_{C_2} = (y_1,\, y_2,\, y_3)$. Thus we have an exact sequence
\begin{equation*}
0 \xlongrightarrow{\hspace*{17pt}} \OOO^\sharp_C \xlongrightarrow{\hspace*{17pt}} \OOO^\sharp_{C_1} \oplus \OOO^\sharp_{C_2} \xlongrightarrow{\hspace*{17pt}} \CC \xlongrightarrow{\hspace*{17pt}} 0.
\end{equation*}
The first part is the corollary to it. The second part is a reformulation of the first.
We note that $\wt_1(\phi)+ \wt_2(\phi) \equiv 0 \mod m$ (see \eqref{eq:kAD:def-wt1-2}). 
\end{proof} 

\begin{scorollary}\label{corollary:xkad.3.18)}
$\gr^1_C \OOO$ has an $\ell$-splitting 
$\gr^1_C \OOO = \AAA \toplus \BBB$ which induces $\ell$-isomorphisms 
$\AAA \totimes \OOO_{C_i} = \AAA_i$ and $\BBB \totimes \OOO_{C_i} = \BBB_i$ for $i\in \{1,\, 2\}$.
\end{scorollary}

In view of Lemma~\ref{lemma:xkad.3.16)}, this follows from Lemma~\ref{lemma:xkad.3.12)q} and
Lemma~\ref{lemma:xkad.3.14)o}. 

Now we disprove~\ref{case:xkad.3.7)gcase-4} by the following.

\begin{proposition}[Case \typec{k3A} with $m=3$]
\label{proop:k3a.1}
The subcase \xref{case:xkad.3.7)gcase-4} does not occur.
\end{proposition}

\begin{proof}
Assume that we are in the situation of~\ref{case:xkad.3.7)gcase-4}.
Since $\gr_C^0 \upomega_X^\sharp$ at~$P^\sharp$
is not generated by an $\mumu_{3}$-invariant, 
the $\ell$-homomorphism 
$\gr_C^0 \upomega_X  \to \gr_{C_1}^0 \upomega_X \toplus \gr_{C_2}^0 \upomega_X$ is an $\OOO_{C}$-module isomorphism. In this sense
\begin{equation}
\label{eq:revision:grC0omega}
\gr_C^0 \upomega_X\simeq\gr_{C_1}^0 \upomega_X \toplus \gr_{C_2}^0 \upomega_X
= (-1+2P_1^\sharp+(m'-a')Q^\sharp) \toplus (-1+P_2^\sharp+R^\sharp),
\end{equation}
(see \cite[(5.15)]{MP20})
and $H^i(C,\, \gr_C^0 \upomega_X)=H^i(C_1,\OOO_{C_1}(-1)^{\oplus 2})=0$ for $i\in \{0,\, 1\}$.
By Lemma~\ref{lemma:xkad.3.14)o}
\begin{eqnarray*}
\AAA_2\totimes \upomega_X &=& \bigl(-1+2P_2^\sharp\bigr),
\\
\BBB_2\totimes \upomega_X &=& \bigl(-1+P_2^\sharp\bigr),
\end{eqnarray*}
$(\AAA \totimes \upomega_X)^\sharp$, $(\BBB \totimes \upomega_X)^\sharp$ at~$P^\sharp$
are not 
generated by $\mumu_m$-invariants. Whence similarly to (\ref{eq:revision:grC0omega}) we have $\OOO_{C}$-isomorphisms 
$\AAA \totimes \upomega_X \simeq  \AAA_1 \totimes \upomega_X \toplus  \AAA_2 \totimes \upomega_X$
and 
$\BBB \totimes \upomega_X \simeq \BBB_1 \totimes \upomega_X \toplus  \BBB_2 \totimes \upomega_X$.
Hence we have the following isomorphisms of $\OOO_{C}$-modules by Corollary~\ref{corollary:xkad.3.18)}.
\begin{eqnarray*}
\gr_{C}^1 \upomega_X 
&=& \AAA \totimes \upomega_X \toplus \BBB \totimes \upomega_X
\\
&\simeq& \underset{i\in \{1,\, 2\}}\btoplus\bigl(\AAA_i \totimes \upomega_X \toplus \BBB_i \totimes \upomega_X\bigr).
\end{eqnarray*}
By Lemma~\ref{lemma:xkad.3.12)q}
\begin{eqnarray*}
\AAA_1\totimes \upomega_X &=& \bigl(-1+P_1^\sharp+2(m'-a')Q^\sharp\bigr),
\\
\BBB_1\totimes \upomega_X &=& \bigl(-1+2P_1^\sharp+(m'-a'+1)Q^\sharp\bigr),
\end{eqnarray*}
we have $H^i(C_j,\, \AAA_j\totimes \upomega_X)=H^i(C_j,\, \BBB_j\totimes 
\upomega_X)=0$
and $H^i(\gr_{C}^1 \upomega_X)=0$ for all $i,\, j$ by $2(m'-a')<m'$ (see Lemma~\ref{lemma:xkad.3.3)a}).

Let $F^n_C (\OOO_X)$ be the $n$-th symbolic power of the defining ideal $\III_C$
of $C$ and $F^n (\upomega_X)= F^n_C (\OOO_X)\totimes \upomega_X$. Then by
\begin{equation}\label{gr^n-omega-to-omega/F^n+1}
0 \xlongrightarrow{\hspace*{17pt}} \gr_C^n \upomega_X \xlongrightarrow{\hspace*{17pt}}
\upomega_X/F_C^{n+1} \upomega_X \xlongrightarrow{\hspace*{17pt}} 
\upomega_X/F_C^{n} \upomega_X \xlongrightarrow{\hspace*{17pt}} 0
\end{equation}
with $n \in \{0,1\}$,
we have $H^i(X,\, \upomega_X/F^1 \upomega_X) = H^i(X,\, \upomega_X/F^2 \upomega_X)=0$ for 
$i\in \{0,\, 1\}$.
Whence we have $H^i(C,\, \gr_C^2 \upomega_X)\simeq H^i(X,\, \upomega_X/F^3 \upomega_X)$ from $n=2$.
The natural $\ell$-homomorphism 
\begin{equation*}
\alpha : \tilde{S}^2(\gr_C^1 \OOO_X) \xlongrightarrow{\hspace*{17pt}} \gr_C^2 \OOO_X \end{equation*}
is an $\ell$-isomorphism on $X \setminus \{S\}$ because $C^\sharp \subset X^\sharp$
is a complete intersection at~$P^\sharp, Q^\sharp, R^\sharp$. At $S$, the length of 
$\coker(\alpha)$ is at most 1 \cite[(2.12.1)]{KM92}.
Hence the length of 
\begin{equation*}
\coker\left[\alpha \totimes \upomega_X : \tilde{S}^2\bigl(\gr_C^1 \OOO_X\bigr)\totimes \upomega_X 
\xlongrightarrow{\hspace*{17pt}} \gr_C^2
\upomega_X\right] \simeq (\coker \, \alpha)\otimes \upomega_X
\end{equation*}
is also at most~$1$,
and 
\begin{equation}
\label{eq:H1(CS2}
\dim H^1\left(C,\, \tilde{S}^2\bigl(\gr_C^1 \OOO_X\bigr)\totimes \upomega_X\right) \le \dim H^1 (\gr^2_{C}\upomega_X)+1. 
\end{equation} 
We have the following by Corollary~\ref{corollary:xkad.3.18)}.
\begin{eqnarray*}
\tilde{S}^2\bigl(\gr_C^1 \OOO_X\bigr)\totimes \upomega_X
&=& \bigl(\AAA^{\totimes 2}\totimes \upomega_X\bigr) \toplus \bigl(\AAA \totimes \BBB \totimes \upomega_X\bigr)
\toplus \bigl(\BBB^{\totimes 2}\totimes \upomega_X\bigr) 
\\
&=& \bigl(\AAA^{\totimes 2}\totimes \upomega_X\bigr) \toplus 
\left(\underset{i\in \{1,\, 2\}}{\btoplus}
\left(\bigl(\AAA_i \totimes \BBB_i \totimes \upomega_X\bigr) 
\toplus \bigl(\BBB_i^{\totimes 2}\totimes\upomega_X\bigr)\right)\right),
\end{eqnarray*}
because we can apply the argument for
$\gr_C^0 \upomega_X$ (\ref{eq:revision:grC0omega})
to $\AAA \totimes \BBB \totimes \upomega_X$ and 
$\BBB^{\totimes 2}\totimes \upomega_X$ by
\begin{eqnarray*}
\AAA_2 \totimes \BBB_2 \totimes \upomega_X &=& \bigl(-2+2P_2^\sharp+R^\sharp\bigr),
\\
\BBB_2^{\totimes 2}\totimes\upomega_X &=& \bigl(-2+P_2^\sharp + R^\sharp\bigr).
\end{eqnarray*}
The last equalities also imply
\begin{equation*}
\h^1\bigl(C,\, \AAA_2 \totimes \BBB_2 \totimes \upomega_X\bigr) = 
\h^1\bigl(C,\, \BBB_2^{\totimes 2}\totimes\upomega_X\bigr)=1.
\end{equation*}
Hence $\h^1\bigl(C,\, \tilde{S}^2(\gr_C^1 \OOO_X)\totimes \upomega_X\bigr) \ge 2$,
which implies $\h^1 (X,\, \upomega_X/F^3 \upomega_X)=\h^1(\gr_C^2 \upomega_X) \ge 1$ by \eqref{eq:H1(CS2}. 
Hence $f : (X,C)\to (Z,o)$ is a $\QQ$-conic bundle by \cite[(1.2.1)]{Mori:flip}
where $(Z,o)$ is smooth by Remark~\ref{rem:Cl}.
Then the closed subscheme of $X$ defined by $F^3 \OOO_X$
contains the whole fiber $f^{-1}(o)$ (see \cite[Theorem~4.4]{MP:cb1}), that is, $\mmm_{Z,o}\OOO_X \supset F_C^3(\OOO_X)$.

By the $\ell$-decomposition of $\gr_C^1 \OOO_X$ \eqref{corollary:xkad.3.18)} and the descriptions of $\AAA_i$'s and $\BBB_i$'s \eqref{lemma:xkad.3.12)q} and \eqref{lemma:xkad.3.14)o}, 
we see that an $\ell$-generator of $\AAA$ (resp. $\BBB$) at $P$ is non-$\ell$-invariant (resp. $\ell$-invariant) and we have the $\OOO_{C}$-module exact sequence
\begin{equation*}
0 
\xlongrightarrow{\hspace*{17pt}}
\gr_C^1 \OOO_X 
\xlongrightarrow{\hspace*{17pt}}
\gr_{C_1}^1 \OOO_X \oplus \gr_{C_2}^1 \OOO_X 
\xlongrightarrow{\hspace*{17pt}}
\CC_P 
\xlongrightarrow{\hspace*{17pt}} 0,
\end{equation*}
which further implies
$H^0(\gr_C^1 \OOO_X)=H^{0}(F^{1}_{C}(\OOO_{X})/F^{2}_{C}(\OOO_{X})) = 0$.

Hence we see 
$\mmm_{Z,o}\OOO_X \subseteq F_C^2(\OOO_X)$, and
let $t_1, t_2$ be a set of generators of $\mmm_{Z,o}$.

\begin{slemma}\label{t1t2-lin-indep}
The images $\bar{t_1}, \bar{t_2}$ of $t_1, t_2$ in $H^0(\gr_C^2 \OOO_X)$ are linearly independent.
\end{slemma}

\begin{proof}
Assume the contrary. We may assume that $\bar{t_2}=0 \in H^0(\gr_C^2 \OOO_X)$ and $t_2 \in H^0(F_C^3 \OOO_X)$.
Let $\eta$ be a generic point of $C$. Then 
$\OOO_{X,\eta}/(t_1,t_2)\OOO_{X,\eta}$ 
is of length $\ge 2 \cdot 3$.
On the other hand, $\OOO_{X,\eta}/F_C^3(\OOO_X)_{\eta}$ is of length $6$, whence 
$F_C^3(\OOO_X)_{\eta}=(t_1,t_2)\OOO_{X,\eta}$. Since $F_C^3(\OOO_X)_{\eta}=I_{C,\eta}^3$, $F_C^3(\OOO_X)_{\eta}$ cannot be a complete intersection ideal. This is a contradiction.
\end{proof}

By the description of $\AAA_i$'s and $\BBB_i$'s \eqref{lemma:xkad.3.12)q} and \eqref{lemma:xkad.3.14)o},
we see
\[
\begin{array}{rclrcl}
\AAA_1^{\totimes 2} &=& (-1 + P_1^\sharp + 2(m'-a')Q^\sharp),& \AAA_2^{\totimes 2} &=& (-1 + 2P_2^\sharp),\\
\BBB_1^{\totimes 2} &=& (2Q^\sharp),& \BBB_2^{\totimes 2} &=& (-1),\\
\AAA_1\totimes \BBB_1 &=& (-1+2P_1^\sharp+(m'-a'+1)Q^\sharp),&\AAA_2\totimes \BBB_2 &=& (-1+P_2^\sharp).\\
\end{array}
\]
Whence $\h^0(\tilde{S}^2(\gr_C^1 \OOO_X))=0$ by \eqref{corollary:xkad.3.18)}.
This however contradicts 
Lemma~\ref{t1t2-lin-indep}, because
$\h^0(\gr_C^2 \OOO_X) \le \h^0(\tilde{S}^2(\gr_C^1 \OOO_X)) + 1$ 
(similarly to \eqref{eq:H1(CS2}).
This proves Proposition~\ref{proop:k3a.1}.
\end{proof}

\subsection{(Case \typec{kAD} with $m\ge 5$)}
The rest of this section will be devoted to
proving Theorem~\ref{thm:KAD-k3A} assuming that $C_2$ is in~\ref{case:xkad.3.7)gcase-1}.
The first part of the following is the same as Corollary~\ref{corollary:xkad.3.18)}.

\begin{sremark}\label{some-formular-in-kAD}
For ease of reference, we collect some formulae in case~\ref{case:xkad.3.7)gcase-1}.
\[
\begin{array}{rclrcl}
\AAA_1 &=& \bigl(-1 + \textstyle\frac{m+1}2 P_1^\sharp +(m'-a')Q^\sharp\bigr),&\AAA_2 &=& \bigl(\textstyle\frac{m-1}2 P^\sharp_2 +R^\sharp\bigr),\\
\BBB_1 &=& (Q^\sharp)&\BBB_2&=&(-1+R^\sharp),\\
\AAA_1^{\totimes 2} &=& (-1 + P_1^\sharp + 2(m'-a')Q^\sharp),& \AAA_2^{\totimes 2} &=& (1 + (m-1)P_2^\sharp),\\
\BBB_1^{\totimes 2} &=& (2Q^\sharp),& \BBB_2^{\totimes 2} &=& (-1),\\
\AAA_1\totimes \BBB_1 &=& (-1+\frac{m+1}{2}P_1^\sharp+(m'-a'+1)Q^\sharp),&\AAA_2\totimes \BBB_2 &=& (\frac{m-1}{2}P_2^\sharp),\\
\end{array}
\]
Moreover, 
for $\gr_{C_1}^0 \upomega_X$ from Lemma~\ref{lemma:xkad.3.10)} and for $\gr_{C_2}^0\upomega_X$ from the dual of 
\cite[(5.20)]{MP20},
we obtain
\begin{equation*}
\gr_C^0 \upomega_X =\gr_{C_1}^0 \upomega_X \toplus \gr_{C_2}^0 \upomega_X
= \left(-1+\textstyle\frac{m+1}{2}P_1^\sharp+(m'-a')Q^\sharp\right) \tOplus \left(-1+\textstyle\frac{m-1}{2}P_2^\sharp+R^\sharp\right).
\end{equation*}
\end{sremark}

Let $\JJJ \subset \OOO_X$ be an ideal such that $\III^{(2)}_C \subset \JJJ \subset \III_C$ and 
$\JJJ/\III^{(2)}_C = \AAA$ in Corollary~\ref{corollary:xkad.3.18)}. Then $\JJJ$ is (1,2)-monomializable at~$P,Q$ and $R$
\cite[(8.10)]{Mori:flip} and it is a nested c.i. on $C\setminus\{P,Q,R\}$ \cite[(8.4)]{Mori:flip} and induces
an $\ell$-exact sequence

\begin{equation}
\label{eq:ref(xkad.3.19).1)}
0 \xlongrightarrow{\hspace*{17pt}}\BBB^{\totimes 2} \xlongrightarrow{\hspace*{17pt}}\JJJ \totimes \OOO_C \xlongrightarrow{\hspace*{17pt}}\AAA \xlongrightarrow{\hspace*{17pt}}0. 
\end{equation} 
We note that $\JJJ$ is a $C$-laminal ideal (i.e. associated primes of $\JJJ$ 
are among $\III_{C_1}$ and $\III_{C_2}$ and $\III^{(2)}_{C_i} \not\supset \JJJ$ for $i\in \{1,\, 2\}$),
and we will use the associated filtrations $F^\bullet(\bullet,\ \JJJ)$ \cite[(8.2)]{Mori:flip}.
We recall two $\ell$-isomorphisms
\begin{align*}
\gr^1(\OOO,\, \JJJ) &\simeq \III_C/\JJJ \simeq \BBB, 
\\
\gr^2(\OOO,\, \JJJ) &\simeq \JJJ \totimes \OOO_C.
\end{align*}
Again by the deformation method, we can reduce the statement of Theorem~\ref{thm:KAD-k3A} to the case, where a certain
$\ell$-splitting is general.

\begin{slemma}[Deformation of patching]
\label{lemma:xkad.3.20)z}
We may assume that we have an $\ell$-splitting
\begin{equation*}
\gr^2(\OOO,\, \JJJ) \totimes \OOO_{C_2} = \DDD_2 \toplus \EEE_2,
\end{equation*}
where $\DDD_2=(0)$ and $\EEE_2 = \bigl(-1+\frac{m-1}2 P^\sharp_2 + R^\sharp\bigr)$.
\end{slemma}

\begin{proof}
The $\ell$-exact sequence \eqref{eq:ref(xkad.3.19).1)}
induces the $\ell$-exact sequence
\begin{equation*}
0 \xlongrightarrow{\hspace*{17pt}}(-1) \xlongrightarrow{\hspace*{17pt}}\gr^2(\OOO,\, \JJJ) \totimes \OOO_{C_2} \xlongrightarrow{\hspace*{17pt}}(\textstyle\frac{m-1}2 P^\sharp_2 + R^\sharp) 
\xlongrightarrow{\hspace*{17pt}}0.
\end{equation*}
Thus $H^0(\gr^2(\OOO,\, \JJJ) \totimes \OOO_{C_2}) \simeq \CC$, and it has a unique non-zero element $s$ up to a constant multiplication.
Since $\gr^2(\OOO,\, \JJJ) \totimes \OOO_{C_2}$ has an $\ell$-splitting of $\ell$-invertible
sheaves \cite[(2.8)]{KM92}, we have one of the following $\ell$-splittings.
\begin{equation*}
\gr^2(\OOO,\, \JJJ) \totimes \OOO_{C_2} =
\begin{cases}
(-1+\frac{m-1}2 P^\sharp_2 + R^\sharp) \toplus (0) 
\\[5pt]
(\frac{m-1}2 P^\sharp_2) \toplus (-1+R^\sharp) 
\\[5pt]
(-1 + \frac{m-1}2 P^\sharp_2) \toplus (R^\sharp) 
\\[5pt]
(\frac{m-1}2 P^\sharp_2 + R^\sharp) \toplus (-1)
\end{cases}
\end{equation*} 
If $s$ does not vanish at~$P^\sharp$ or $R^\sharp$ (i.e. the images of $s$
in $\gr^2(\OOO,\, \JJJ)^\sharp \otimes \CC_{P^\sharp}$ at~$P^\sharp$ and
$\gr^2(\OOO,\, \JJJ)^\sharp \otimes \CC_{R^\sharp}$ at $R^\sharp$ are both non-zero),
then we are in the first case.

If $X \supset C$ is not in the first case, we construct a 
deformation $X_t \supset C_t$ ($t \in (\CC,0)$) of $X \supset C$ such
that $X_t \supset C_t$ ($t\not=0$) is in the first case. As in Lemma 
\ref{lemma:xic.3.8)}, this proves Lemma~\ref{lemma:xkad.3.20)z}.
To be precise, we will construct here a flat deformation $X_t \supset C_t$
of $X \subset C$ such that

\begin{enumerate}
\item 
\label{case:xkad.3.21)-i}
there exist open neighborhoods $U_1$ and $V_1$ of $C_1$, and $U_2$
of $C_2\setminus C_2 \cap V_1$
such that $V_1$ contains the closure of $U_1$, $U_1 \cap U_2=\emptyset$ and the deformation is trivial on $U_1$ and $U_2$,
\item 
\label{case:xkad.3.21)-ii}
the family is trivial modulo $\III^{(2)}_C$, and hence $\JJJ$ deforms in a 
flat family $\{\JJJ_t\}_t$ and
$\gr^2(\OOO,\, \JJJ)$ in $\{\gr^2(\OOO,\, \JJJ_t)\}_t$, and
\item 
\label{case:xkad.3.21)-iii}
the global section $s$ of $\gr^2(\OOO,\, \JJJ)$ chosen above deforms to the global
section $s_t$ of $\gr^2(\OOO,\, \JJJ_t)$ such that $s_t$ does not vanish at~$P^\sharp$.
\end{enumerate}

By its construction we have $\JJJ^\sharp = (y'_2,(y_1y_3)^2)$, where $y'_2 \in \OOO_{X^\sharp}$
is a $\mumu_m$-semi-invariant such that $y'_2 \equiv y_2 \mod (y_1y_3) \cdot \OOO_{X^\sharp}$.
We may assume that $s$ vanishes at~$P^\sharp$ because otherwise we just take the
trivial deformation. Hence $s$ is represented in a neighborhood of $P$ by $\tilde{s} \in \OOO_X$
such that
\begin{equation*}
\tilde{s} \equiv y_1^{\frac{m-1}2}y'_2 \cdot(\text{unit}) \mod y_1^m(y_1y_3)^2.
\end{equation*}
Let $u_1=\tilde{s}+t \cdot (y_1y_3)^2$ and $u_2=y_1y_3$. We use the twisted extension with
$u=(u_1,\, u_2)$ \cite[(1b.8.1)]{Mori:flip} to $X \supset C_2$
(cf. also \cite[(9.6)]{Mori:flip}). Since we are only twisting the 
patching,~\ref{case:xkad.3.21)-i} is satisfied by the construction. The deformation
is trivial modulo $\III^{(2)}_C$ because $u \equiv u(0) \mod \III^{(2)}_C$,
whence~\ref{case:xkad.3.21)-ii}. For~\ref{case:xkad.3.21)-iii}, we note that
\begin{equation*}
s_t =
\begin{cases}
s+t(y_1y_3)^2 &\text{in $U_1$} 
\\
s &\text{in $U_2$} 
\end{cases}
\end{equation*}
is the global section extending $s$. Then~\ref{case:xkad.3.21)-iii} is obvious. 
This takes care of
the non-vanishing of $s$ at~$P^\sharp$. If $s$ vanishes at $R^\sharp$, a similar construction applies and hence
we omit it here. This completes the proof of Lemma~\ref{lemma:xkad.3.20)z}.
\end{proof} 

\begin{slemma}\label{lemma:xkad.3.22)}
We have an $\ell$-splitting
$\gr^2(\OOO,\, \JJJ) \totimes \OOO_{C_1} = \DDD_1 \toplus \EEE_1$
with $\ell$-isomorphisms
\begin{align*}
\DDD_1 &\simeq \BBB_1^{\totimes 2} = (2Q^\sharp), 
\\
\EEE_1 &\simeq \AAA_1 = \bigl(-1 + \textstyle\frac{m+1}2 P^\sharp_1 + (m'-a') Q^\sharp\bigr).
\end{align*}
\end{slemma}

\begin{proof}
By
\begin{equation*}
\BBB_1^{\totimes 2} \totimes \AAA_1^{\totimes (-1)} = \bigl(-1 + \textstyle\frac{m-1}2 P^\sharp_1 + (a'+2)Q^\sharp\bigr),
\end{equation*}
we have $H^1\bigl(C_1,\, \BBB_1^{\totimes 2} \totimes \AAA_1^{\totimes (-1)}\bigr) = 0$. Thus by
Proposition~\ref{extention-class},
the $\ell$-sequence $\text{\eqref{eq:ref(xkad.3.19).1)}} \totimes \OOO_{C_1}$ is $\ell$-split. 
\end{proof}

\begin{sdefinition}\label{def:xkad.3.24)}
In view of Lemma~\ref{lemma:xkad.3.16)}, there is an
$\ell$-splitting
$\gr^2(\OOO,\, \JJJ) = \DDD \toplus \EEE$, where $\DDD$ and $\EEE$ are $\ell$-invertible sheaves
on $C$ such that $\DDD \totimes \OOO_{C_i} = \DDD_i$ and $\EEE \totimes \OOO_{C_i} = \EEE_i$
with $\DDD_i$ and $\EEE_i$ given in Lemmas~\ref{lemma:xkad.3.20)z} and~\ref{lemma:xkad.3.22)}.
\end{sdefinition} 

\begin{proof}[Proof of Theorem~\xref{thm:KAD-k3A}] 
We note $\ell$-isomorphisms
\begin{align*}
\upomega_X \totimes \OOO_{C_1} &= \bigl(-1 + \textstyle\frac{m+1}2 P^\sharp_1 + (m'-a') Q^\sharp\bigr), 
\\
\upomega_X \totimes \OOO_{C_2} &= \bigl(-1 + \textstyle\frac{m-1}2 P^\sharp_2 + R^\sharp\bigr).
\end{align*}
Thus 
\begin{equation*}
H^{i}(\gr^0(\upomega,\, \JJJ))=H^{i}(\upomega \totimes \OOO_C) = 0, \qquad i\in \{0,\, 1\}
\end{equation*}
(cf. proof of Proposition~\ref{proop:k3a.1}).
We have the $\ell$-exact
sequence ($n \ge 0$)
\begin{equation}
\label{gr^n^omega-J}
0 \xlongrightarrow{\hspace*{17pt}} \gr^n(\upomega,\JJJ) \xlongrightarrow{\hspace*{17pt}}
\upomega_X/F^{n+1}(\upomega,\JJJ) \xlongrightarrow{\hspace*{17pt}} 
\upomega_X/F^{n}(\upomega,\JJJ) \xlongrightarrow{\hspace*{17pt}} 0
\end{equation}
similar to \eqref{gr^n-omega-to-omega/F^n+1}
and $\gr^0(\upomega,\JJJ)=\upomega_X/F^{1}(\upomega,\JJJ)$.
From this sequence with $n=0$,
we have $H^i\bigl(\upomega_X/F^1(\upomega,\, \JJJ)\bigr)=0$ for $i\in \{0,\, 1\}$. We note $\ell$-isomorphisms
\begin{align*}
\gr^1(\upomega,\, \JJJ) &= \upomega \totimes \gr^1(\OOO,\, \JJJ) = \upomega \totimes \BBB, 
\\
\upomega \totimes \BBB_1 &= \bigl(-1 + \textstyle\frac{m+1}2 P^\sharp_1 + (m'-a'+1)Q^\sharp\bigr), 
\\
\upomega \totimes \BBB_2 &= \bigl(-1 + \textstyle\frac{m-1}2 P^\sharp_2\bigr), 
\end{align*}
where we note $m'-a'+1 < m'$ by Lemma~\ref{lemma:xkad.3.3)a}. Thus $H^i(\gr^1(\upomega,\, \JJJ))=0$ for $i\in \{0,\, 1\}$,
and by \eqref{gr^n^omega-J} with $n=1$,
we have
\begin{equation}\label{eq:refcase:xkad.3.25).1}
H^i\bigl (\upomega/F^2(\upomega, \JJJ)\bigr) = 0, \qquad i\in \{0,\, 1\}.
\end{equation} 

We will derive a contradiction out of this by a further thickening.

Let $\MMM$ be the ideal sheaf of $\OOO_X$ such that
$\JJJ \supset \MMM \supset F^3(\OOO,\, \JJJ)$ and $\MMM/F^3(\OOO,\, \JJJ)=\DDD$ (cf. (\ref{def:xkad.3.24)})). 
Since
$\JJJ$ is (1,2)-monomializable at~$P,\, Q$ and $R$ and a nested c.i.
on $C\setminus\{P,Q,R\}$ (Corollary~\ref{corollary:xkad.3.18)}), we have an $\ell$-isomorphism
\begin{equation*}
\gr^2(\OOO,\, \JJJ) \totimes \BBB \simeq \gr^3(\OOO,\, \JJJ),
\end{equation*}
(cf. \cite[(8.10)]{Mori:flip}). Let $\MMM'$ 
be the ideal sheaf of $\OOO_X$
such that 
\begin{equation*}
\begin{array}{l}
F^3(\OOO,\, \JJJ) \supset \MMM' \supset F^4(\OOO,\, \JJJ),
\\[7pt]
\MMM'/F^4(\OOO,\, \JJJ)=\DDD \totimes \BBB.
\end{array}
\end{equation*}
Then
$\MMM\cdot \III_C \subset \MMM'$, $\MMM/\MMM'$ is a locally $\ell$-free $\OOO_C$-module,
and we have an $\ell$-exact sequence
\begin{equation}
\label{eq:refcase:xkad.3.25).2)}
0 \xlongrightarrow{\hspace*{17pt}}\EEE \totimes \BBB \xlongrightarrow{\hspace*{17pt}}\MMM/\MMM' \xlongrightarrow{\hspace*{17pt}}\DDD \xlongrightarrow{\hspace*{17pt}}0.
\end{equation} 
The computations are performed by the following scheme:
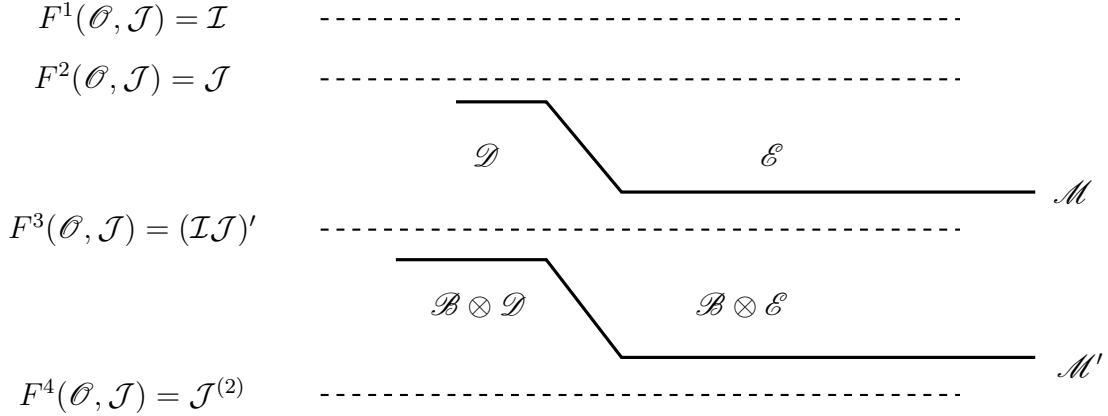
\[
\begin{tikzpicture}
\draw[black, thick, dashed] (-5,4.8) -- (3.5,4.8);
\draw (-7.5,4.8) node {$F^1(\OOO,\JJJ)=\III$};
\draw[black, thick, dashed] (-5,4) -- (3.5,4);
\draw (-7.5,4) node {$F^2(\OOO,\JJJ)=\JJJ$};

\draw[black, very thick] (-3.2,3.7) -- (-2,3.7) -- (-1,2.5)-- (4.5,2.5);
\draw[black, thick, dashed] (-5,2) --(3.5,2);
\draw (-7.5,2) node {$F^3(\OOO,\JJJ)= (\III\JJJ)'$};

\draw[black, very thick] (-4.,1.6) -- (-2,1.6) -- (-1,0.3)-- (4.5,0.3);
\draw[black, thick, dashed] (-5,-0.2) -- (3.5,-0.2);
\draw (-7.5,-0.2) node {$F^4(\OOO,\JJJ)= \JJJ^{(2)}$};

\draw (-2.8,3) node {$\DDD$};
\draw (1.,3) node {$\EEE$};
\draw (-2.9,1) node {$\BBB\otimes\DDD$};
\draw (0.6,1) node {$\BBB\otimes\EEE$};
\draw (5,2.5) node {$\MMM$};
\draw (5.1,0.2) node {$\MMM'$};
\end{tikzpicture}
\]
We note
\begin{align*}
\EEE_1 \totimes \BBB_1 \totimes \DDD_1^{\totimes (-1)} 
&= \bigl(-1 + \textstyle\frac{m+1}2 P_1^\sharp + (m'-a'-1) Q^\sharp\bigr),
\\
\EEE_2 \totimes \BBB_2 \totimes \DDD_2^{\totimes (-1)} 
&= \bigl(-1 + \textstyle\frac{m-1}2 P_2^\sharp\bigr),
\end{align*}
whence $H^1\bigl(C,\, \EEE\totimes \BBB \totimes \DDD^{\totimes (-1)}\bigr)=0$ 
and
\eqref{eq:refcase:xkad.3.25).2)} is $\ell$-split \cite[(2.6.2)]{KM92}. Thus
\begin{equation*}
\label{eq:refcase:xkad.3.25).3)}
\MMM/\MMM' = (\EEE \totimes \BBB) \toplus \DDD.
\end{equation*} 
\begin{scase}
Let $\NNN$ be the ideal sheaf of $\OOO_X$ such that
$\MMM \supset \NNN \supset \MMM'$ and $\NNN/\MMM'=\DDD$.
The computations here are arranged as follows:
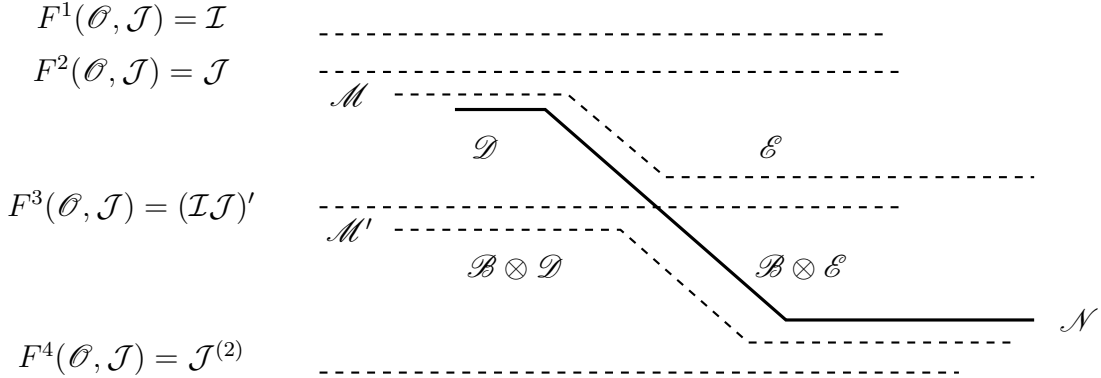
\[
\begin{tikzpicture}
\draw[black, thick, dashed] (-5,4.8) -- (2.5,4.8);
\draw (-7.5,5) node {$F^1(\OOO,\JJJ)=\III$};
\draw[black, thick, dashed] (-5,4.3) -- (2.7,4.3);
\draw (-7.5,4.3) node {$F^2(\OOO,\JJJ)=\JJJ$};
\draw[black, thick, dashed] (-4,4) -- (-1.7,4) -- (-0.4,2.9)-- (4.5,2.9);
\draw (-4.6,4) node {$\MMM$};

\draw[black, very thick] (-3.2,3.8) -- (-2,3.8) -- (1.2,1)-- (4.5,1);
\draw[black, thick, dashed] (-5,2.5) --(2.7,2.5);
\draw (-7.5, 2.5) node {$F^3(\OOO,\JJJ)= (\III\JJJ)'$};

\draw[black, thick, dashed] (-4,2.2) -- (-1,2.2) -- (0.7,0.7)-- (4.2,0.7);
\draw (-4.6,2.2) node {$\MMM'$};

\draw[black, thick, dashed] (-5,0.3) -- (3.5,0.3);
\draw (-7.5,0.5) node {$F^4(\OOO,\JJJ)= \JJJ^{(2)}$};

\draw (-2.8,3.3) node {$\DDD$};
\draw (1,3.3) node {$\EEE$};

\draw (-2.4,1.7) node {$\BBB\otimes\DDD$};
\draw (1.4,1.7) node {$\BBB\otimes\EEE$};

\draw (5.1,1.) node {$\NNN$};
\end{tikzpicture}
\]
We note $H^i(\upomega \totimes \EEE) = 0$ for $i\in \{0,\, 1\}$ by
\begin{align*}
\upomega \totimes \EEE_1 &= \bigl(-1 + P^\sharp_1 + 2(m'-a')Q^\sharp\bigr), 
\\
\upomega \totimes \EEE_2 &= \bigl(-1 + (m-1) P^\sharp_2\bigr),
\end{align*}
where $2(m'-a') < m'$ (Lemma~\ref{lemma:xkad.3.3)a}). Thus by the $\ell$-exact sequence
\begin{equation*}
0 \xlongrightarrow{\hspace*{17pt}}
\upomega_X \totimes \EEE 
\xlongrightarrow{\hspace*{17pt}}
\upomega_X/\upomega \totimes \MMM 
\xlongrightarrow{\hspace*{17pt}}
\upomega_X/F^2(\upomega,\, \JJJ) 
\xlongrightarrow{\hspace*{17pt}}0
\end{equation*}
and $H^i(\upomega_X/F^2(\upomega,\, \JJJ)) = 0$ for $i\in \{0,\, 1\}$ \eqref{eq:refcase:xkad.3.25).1}, we have
$H^i(\upomega_X/\upomega_X \totimes \MMM) = 0$ for $i\in \{0,\, 1\}$. 
By the $\ell$-isomorphism $\MMM/\NNN \simeq \EEE \totimes \BBB$
and 
\begin{equation*}
\upomega \totimes \EEE \totimes \BBB_2 = \bigl(-2 + (m-1) P^\sharp_2 + R^\sharp\bigr),
\end{equation*}
we have
$H^1\bigl(\upomega \totimes \EEE \totimes \BBB_2\bigr) \not = 0$
and $H^1(\upomega_X/\upomega_X\totimes \NNN)\not = 0$, by
\begin{equation*}
0 \xlongrightarrow{\hspace*{17pt}}
\upomega_X \totimes \MMM/\NNN
\xlongrightarrow{\hspace*{17pt}}
\upomega_X/\upomega \totimes \NNN
\xlongrightarrow{\hspace*{17pt}}
\upomega_X/\upomega \totimes \MMM
\xlongrightarrow{\hspace*{17pt}}0.
\end{equation*}
As in the proof of Proposition~\ref{proop:k3a.1}, 
this implies that
$f : (X,C) \to (Z,o)$ is a $\QQ$-conic bundle
by \cite[(1.2.1)]{Mori:flip},
where $(Z,o)$ is smooth by Remark~\ref{rem:Cl}.
Then $(t_1,t_2)\OOO_X \supset \NNN$
for some generators $t_1,t_2$ of $\mmm_{Z,o}$ (see \cite[Theorem~4.4]{MP:cb1}).
Let $\eta_i$ be the generic point of $C_i$ for $i\in \{1,\, 2\}$. In view of the filtration
$\NNN \subset \MMM \subset \JJJ \subset \III_C$, we see that $\OOO_X/\NNN$ is of length $4$ at $\eta_i$
for each $i$. Thus we have $\mult_{\eta_i} t_1 \cdot \mult_{\eta_i} t_2 \le 4$ for each $i\in \{1,\, 2\}$.
Since we know by Remark~\ref{some-formular-in-kAD}
\begin{equation*}
H^0(\gr_C^1 \OOO)= H^0(\AAA)\oplus H^0(\BBB)=H^0(\AAA)=H^0(\AAA_2)\simeq \CC,
\end{equation*}
$\mult_{\eta_1} t_j \ge 2$ for each $j$, whence $\mult_{\eta_1} t_j = 2$ for each $i$.
Furthermore by $\h^0(\gr_C^1 \OOO)=1$, we may assume that $t_2\OOO_X \in F_C^2(\OOO_X)$.
By Remark~\ref{some-formular-in-kAD}
\begin{equation*}
H^0(\gr_C^2 \OOO)= H^0(\AAA^{\totimes 2})\oplus H^0(\AAA\totimes\BBB)\oplus H^0(\BBB^{\totimes 2})
=H^0(\AAA_2^{\totimes 2})\oplus H^0(\AAA_2\totimes\BBB_2),
\end{equation*}
we see that the image of $t_2$ in $\gr_C^2 \OOO \totimes \OOO_{C_1}$ is zero, whence $\mult_{\eta_1} t_2 \ge 3$, a contradiction. 
\end{scase}
The subcase~\ref{case:xkad.3.7)gcase-1} is now disproved, and Theorem~\ref{thm:KAD-k3A} is proved.
\end{proof}

\section{Proof of Theorem~\ref{thm:main}}
\label{sect:proof}
In this section we prove Theorem~\ref{thm:main}.
We start with a lemma.

\begin{lemma}
\label{lemma:components}
Let $(X,C)$ be an extremal curve germ with reducible central curve
$C=\cup_{i=1}^{N} C_i$, where $C_i$ are irreducible components.
Assume that there is a component, say $C_1\subset C$, such that 
$(X,C_1)$
is of type~\typec{cD/3}, \typec{IC}, \typec{kAD}, \typec{II^\vee}, \typec{IIB} or~\typec{k3A}.
Furthermore, assume that $X$ is Gorenstein outside of $C_1$ and $C_1$ meets all other components of $C$.
Then the following assertions hold.
\begin{enumerate}
\item\label{lemma:div0}
$N\le 5$ and $N\le 4$ if $(X,C)$ is birational.

\item\label{lemma:div3}
If $(X,C_1)$ is divisorial and is not of type \typec{II^\vee}, then $N\le 3$.

\item\label{lemma:div4}
If $(X,C_1)$ is divisorial, is not of type \typec{II^\vee} and $(X,C)$ is also divisorial,
then $N=2$.

\item\label{lemma:div1}
If $(X,C)$ is flipping, then $N=2$.
\end{enumerate}
\end{lemma}

\begin{proof}
Note that the components of $C$ do not meet each other outside $C_1$ (because $\p(C)=0$).

Assume that $(X,C_1)$ is divisorial. Let $\varphi: X\to X'$ be the corresponding contraction,
let $C':=\varphi(C)$, and let $P':=\varphi(C_1)$.
Note that the point $P'\in X'$ is terminal (see Corollary~\ref{cor:divI}).
Each curve $C_i':=\varphi(C_i)$, $i=2,\dots,N$ passes through the point $P'$ and
is $K_{X'}$-negative by Lemma~\ref{lemma:div}.
If $C_1$ is not of type \typec{II^\vee}, then
$P'\in X'$ is Gorenstein terminal by Corollary~\ref{cor:divI}.
Hence $(X',C')$ is a Gorenstein extremal curve germ
(see Definition~\ref{def1}).
Then $N\le 3$ and $N=2$ if $(X,C)$ is birational (see \cite{Cutkosky-1988} and \cite[Theorem~4.7]{KM92}).
This proves~\ref{lemma:div3} and ~\ref{lemma:div4}.
If $C_1$ is of type \typec{II^\vee}, then $P'\in X'$ is a terminal point of index $2$ again by Corollary~\ref{cor:divI}.
Then the number of components of $C'$ is at most $3$ if $(X',C')$ 
is divisorial (see \cite[Theorem~4.7]{KM92})
and
at most $4$ if $(X',C')$ is $\QQ$-conic bundle germ (see \cite[Theorem~12.1]{MP:cb1}).
This proves 
\ref{lemma:div0} in the case where $(X,C_1)$ is divisorial.

Assume that the germ $(X,C_1)$ is flipping. 
Then $(X,C_1)$ is not of type \typec{II^\vee} \cite[Theorem~4.5]{KM92}, hence it is primitive.
Let $\chi : X\dashrightarrow X^+$ be the corresponding flip.
Let $C_1^+$ be the flipped curve and let
$C'^+\subset X^+$ be the proper transform of $C-C_1$.
Let us denote the components of $C'^+$ that does not pass through $\Sing^{\mathrm{nG}}(X^+)$ by
$A_1,\dots,A_l$, and 
the components of $C'^+$ that pass through $\Sing^{\mathrm{nG}}(X^+)$ we denote by by $B_1,\dots,B_k$.
Thus $l+k=N-1$.
Note that the curves $A_1,\dots,A_l,B_1,\dots,B_k$ do not meet each other outside $C_1^+$
and each of them meets $C_1^+$ at a single point.
Moreover, each $(X^{+}, A_i)$ and each $(X^{+}, B_j)$ is an extremal curve germ
by Lemma~\ref{lemma:flip}.

If $X^+$ is Gorenstein, then $k=0$, each germ $(X^+, A_i)$ is divisorial 
and the curves $A_1,\dots,A_l$
are disjoint by \cite[Theorem~4.7]{KM92}.
Moreover, $K_{X^+}\cdot C_1^+=1$ (see Table~\ref{table2} in Appendix~\ref{sect:app}). 
There exists a contraction $\varphi: X^+\to X'$ of $A_1\cup\dots\cup A_l$, where $X'$ is also Gorenstein.
Let $C_1':=\varphi(C_1^+)$.
Then by Lemma~\ref{lemma:div} we have 
\begin{equation}
\label{eq:newKK}
K_{X'}\cdot C_1'\le K_{X^+}\cdot C_1^+-l=1-l. 
\end{equation} 
If $K_{X'}\cdot C_1'\ge 0$, then $l\le 1$ and $N\le 2$. Let
$K_{X'}\cdot C_1'<0$. Then 
$(X',C'_1)$  must be a Gorenstein extremal curve germ and 
$K_{X'}\cdot C_1'=-1$  or $-2$ (see \cite{Cutkosky-1988}). 
Hence $l\le 3$ and $N\le 4$. This proves~ \ref{lemma:div0} in the case where $X^+$ is Gorenstein.

Thus we may assume that $X^+$ is not Gorenstein. By the classification of flips (see Table~\ref{table2} in Appendix~\ref{sect:app})
the variety
$X^+$ has a unique non-Gorenstein point, say $P^+$, and this point is of index $2$.
If $k>1$, then 
\[
\bigcap _{j=1}^k B_j=\{P^+\}. 
\]
Moreover, $K_{X^+}\cdot C_1^+=1/2$.
As above, each $(X^+, A_i)$ is a Gorenstein divisorial extremal curve germ 
and the curves $A_1,\dots,A_l$ are disjoint.
Hence, similar to \eqref{eq:newKK} by Lemma~\ref{lemma:div} and \cite[(2.3.2)]{Mori:flip} we have
\[
-1\le K_{X^+}\cdot C_1^+ -l =\textstyle\frac12-l 
\]
and so $l\le 1$.
On the other hand, $(X^+, \cup_{j=1}^k B_j)$ 
is a birational extremal curve germ of index $2$.
Therefore, $k\le 3$ by \cite[Theorem~4.7]{KM92} and so $N\le 5$.
Assume that $N=5$. Then $l=1$ and $k=3$. As above, let $\varphi: X^+\to X'$
be the contraction of $A_1$, let $B_j':=\varphi(B_j^+)$, and let $C_1':=\varphi(C_1^+)$.
Then by Lemma~\ref{lemma:div} we have $K_{X'}\cdot C_1'< K_{X^+}\cdot C_1^+ -1<0$ and 
$K_{X'}\cdot B_j'=K_{X^+} \cdot B_j^+<0$.
The curves $C_1',B_1',\dots,B_k'$ pass through the point $\varphi(P^+)$.
Therefore, 
$(X',C'_1\cup B_1'\cup\dots\cup B_k')$ is an extremal curve germ 
of index $2$. 
If it is birational, then $k+1\le 3$ by \cite[Theorem~4.7]{KM92}. This contradicts our 
assumption and so $N= 4$. 
This proves~\ref{lemma:div0}.

Finally, assume that $(X,C)$ is flipping. Then, as above, $l=0$ and $(X^+,C'^+= \cup _{j=1}^k B_j)$ is a flipping extremal curve germ of index $2$. By
\cite[Theorem~4.2]{KM92} $C'^+$ is irreducible,
hence $N=2$. This proves~\ref{lemma:div1}.
\end{proof}

\begin{sremark}
\label{rem:components}
Similar arguments show that a flipping extremal curve germ 
of index $3$ has at most $2$ components. Indeed, in the above notation, each germ $(X,C_i)$ is primitive, 
the point $P$ of index $3$ is unique and all the components $C_i\subset C$ pass through $P$.
For the minimal discrepancies of the flip $X^+$ of $C_1$ we have 
$\dis(X^+)> \dis(X)=1/3$ (see \cite[2.13]{Shokurov:nV} or \cite[Proposition~5.1.11(3)]{KMM}).
Therefore, $\operatorname{Index}(X^+)\le 2$ by \cite{Kawamata:discr}. Note that $X^+$ cannot have 
more that one point of index $2$
on $C^+_1$ because $(X^+,C^+_1)$ is primitive (see Lemma~\ref{lemma:flips2-2}\ref{lemma:flips2-2a}
and Corollary~\ref{corImp}).
Then the assertion follows from \cite[Theorem~4.2]{KM92}.
\end{sremark}

Now we consider the possibilities for $(X,C)$.

\subsection{}
Suppose that $X$ is Gorenstein. Then $(X,C)$ is not flipping \cite[Theorem~6.2(i)]{Mori:flip}.
Since $C$ is reducible, $(X,C)$ cannot be divisorial \cite[Theorem~4.7.2]{KM92}.
We obtain the case~\ref{thm:main:Gor}. 

\subsection{}
From now on we assume that $X$ is not Gorenstein.
According to \cite[Corollary~1.15]{Mori:flip} and \cite[Lemma~4.4.2]{MP:cb1} (see also \cite[Proposition~4.2]{Kollar-1999-R}) each intersection point $C_i\cap C_j$ is non-Gorenstein on $X$.
Since $C$ is connected, each germ $(X,C_i)$ has a non-Gorenstein point.
By \cite[Lemma~2.2]{MP20} each germ $(X,C_i)$ is birational and 
by \cite[Theorem~6.2]{Mori:flip} it has at most two non-Gorenstein points.

\subsection{}
\label{pf:1point}
Suppose that for each component $C_i\subset C$ the germ $(X,C_i)$ has exactly one non-Gorenstein point. 
In this situation the whole germ $(X,C)$ has exactly one non-Gorenstein point, say $P$, and 
all the components $C_i$ pass through $P$. 
Since each germ $(X,C_i)$ is birational,
it cannot be of type~\typec{IE^\vee} or \typec{ID^\vee} \cite[5.3.1]{MP:cb1}, \cite[4.2]{Mori:flip}.
Hence $(X,C_i)$ is one of the following types: 
\typec{k1A}, \typec{cD/2}, \typec{cAx/2}, \typec{cE/2}, \typec{cD/3}, 
\typec{IIA}, \typec{II^\vee}, \typec{IC} or \typec{IIB}.

If the point $P\in X$ is of the exceptional type~\type{cAx/4}, then by 
definition each component
$C_j\subset C$ is one of the types \typec{IIB}, \typec{II^\vee} or \typec{IIA}.

\begin{scase}
Suppose that every component
$C_i\subset C$ is of type~\typec{IIA}. 
The germ $(X,C)$ can be either flipping, divisorial or $\QQ$-conic bundle.

Let us prove that $N \le 7$. Recall that in this case $X$ has a unique non-Gorenstein point,
say $P$ and all the components of $C$ pass through $P$. 
The \type{cAx/4}-point $(X,P)$ in some local coordinates $y_1,\dots,y_4$ has the form:
\[
(X, P ) =(X^\sharp, P^\sharp )/ \mumu_4 = \big(\{\alpha(y_1,\dots,y_4)=0\}\subset \CC^4\big)/ \mumu_4(1, 1, 3, 2),
\]
where $\alpha$ is a semi-invariant of weight $2$,
\[
\alpha\equiv \beta(y_1, y_2)+y_3^2\mod (y_1,\dots,y_4)^3,
\]
and~$\beta$ is a quadratic
form in~$y_1$,~$y_2$ (see \cite{Mori:sing} or \cite{Reid:YPG}). Moreover we may assume that 
$\beta = y_1y_2$ or $y_1^2$. Furthermore, each component $C_i^\sharp\subset C^\sharp$ 
is
the image of the map 
\[
t\longmapsto \big(t\gamma_1^{(i)},\, t\gamma_2^{(i)},\, t^3\gamma_3^{(i)},\, t^2\gamma_4^{(i)}\big)
\]
for some $\gamma_j^{(i)}\in \CC\{t^4\}$ \cite[4.2]{Mori:flip}.
Then we see either
$\gamma_1^{(i)}(0) = 0$ and $\gamma_2^{(i)}(0)\neq 0$ or $\gamma_1^{(i)}(0) \neq 0$ and $\gamma_2^{(i)}(0) = 0$ since $C_i^\sharp$ is smooth and $\beta\big(\gamma_1^{(i)}, \, 
\gamma_2^{(i)}\big)t^2\in (t)^3$.
For $i$ with $\gamma_1^{(i)}(0) = 0$, the monomials $y_2^4$,
$y_2y_3$, $y_2^2y_4$, and $y_4^2$ are the only ones whose expansion in $t$ can start with $t^4$.
Similarly, for $i$ with $\gamma_2^{(i)}(0) = 0$, the monomials $y_{1}^4$,
$y_1y_3$, $y_1^2y_4$, and $y_4^2$ are the only ones whose expansion in $t$ can start with $t^4$.
Hence, for any $i$, the monomials $y^4_2$, $y_2y_3$, $y^2_2 y_4$, $y^2_4$, $y^4_1$, $y_1y_3$, $y^2_1 y_4$ are the only ones whose expansion in $t$ can start with $t^4$.
Since $t^4$ is an uniformizing parameter on $C$, 
this means that the the tangent space $T_{C, P}=\sum T_{C_i, P}$ is contained in the 
$\CC^7$ with coordinates
$y^4_2$, $y_2y_3$, $y^2_2 y_4$, $y^2_4$, $y^4_1$, $y_1y_3$, $y^2_1 y_4$.
On the other hand, since the curve $C=\cup_{i=1}^N C_i$
has arithmetic genus $0$ (see \cite[1.2.1]{Mori:flip}), the tangent spaces of its components $T_{C_i, P}$ are linear independent 
in $T_{C, P}$. Whence we have $N \le 7$.

Assume that $(X,C)$ is flipping. Let $\chi : X\dashrightarrow X^+$ be the flip of $C_1$.
Let $C_1^+$ be the flipped curve and let 
$C_i^+\subset X^+$ be the proper transform of $C_i$ for $i\neq 1$.
The set of non-Gorenstein points of $X^+$ consists of either
a unique point $P^+$ of index $2$ or two points $P^+$ and $Q^+$ of indices $2$ and $3$,
respectively (see Table~\ref{table2} in Appendix~\ref{sect:app}). The former case is treated 
similarly to the proof of Lemma~\ref{lemma:components}. Thus we assume that we are in the latter case, i.e. 
$\Sing^{\mathrm{nG}}(X^+)=\{P^+,\, Q^+\}$.
Each $(X^+, C_i^+)$, $i=2,\dots, N$ is a flipping extremal curve germ.
Hence $C_i^+$ passes either through $P^+$ or $Q^+$.
Then the bound $N\le 4$ follows from \cite[Theorem~4.2]{KM92} and Remark~\ref{rem:components}.
We obtain the case~\ref{thm:main:IIA}. 
\end{scase}

\begin{scase}
\label{case:IIvee-a}
Suppose that two components, say $C_1$ and $C_2$ are of type~\typec{II^\vee}.
Consider the germ $(X'=X, C'=C_1\cup C_2)$. 
By the local description of type~\typec{II^\vee} points we have $-K_{X'}\cdot C_i=1/2$ for $i\in \{1,\, 2\}$
(see \cite[(2.3.2) and Theorem~4.9]{Mori:flip}).
Let $H_i$ be a small disc that meets $C_i$ transversely at a general point.
Then $D:=2K_{X'}+H_1+H_2$ is a $2$-torsion element in $\Cl(X')$
(because $2K_{X'}+H_1+H_2$ a numerically trivial Weil divisor on $X'$ which is not Cartier).
The element $D$ defines a double cover $\pi:(X^\flat, C^\flat) \to (X', C')$ which is \'etale outside of $P$ 
and such that the inverse images of $C_1$ and $C_2$ split \cite[(1.11)]{Mori:flip}.
Thus $(X^\flat, C^\flat)$ is an extremal curve germ of index $2$ whose central fiber has four components.
By \cite[Theorem~4.7]{KM92} 
the germ $(X^\flat, C^\flat)$ cannot be birational.
Hence $(X^\flat, C^\flat)$ and $(X', C')$ are $\QQ$-conic bundle germs.
By \cite[Lemma 2.2]{MP20} we have $C=C'$. Thus $(X,C)$ is a $\QQ$-conic bundle germ over a singular base 
(because $\Cl(X)$ has a $2$-torsion, see e.g.~\cite[Construction~4.5]{MP-1p}). Then by \cite{MP:cb2} we obtain the case~\ref{thm:main:IIdual2}.
\end{scase}

\begin{scase}
\label{case:IIvee}
Suppose that exactly one component, say $C_1$ is of type~\typec{II^\vee}.
By Lemma~\ref{lemma:IIvee-IIB} below other components $C_i$, $i\neq 1$ are of type~\typec{IIA}.
The restriction on the number of component follows from Lemma~\ref{lemma:components}.
We obtain the case~\ref{thm:main:IIdual}.

\begin{slemma}
\label{lemma:IIvee-IIB}
If $C_1$ is of type \typec{II^\vee}, then $C_2$ cannot be of type \typec{IIB}.
\end{slemma}

\begin{proof}
Assume that $C_1$ is of type \typec{II^\vee} and $C_2$ is of type \typec{IIB}. 
Note that in the $\QQ$-conic bundle case the base $(Z,o)$ must be smooth (see e.g. \cite[Theorem~1.3]{MP:cb2}).
Let $C_1\cap C_2=\{P\}$, let $\pi:(X^\sharp,P^\sharp)\to (X,P)$ be the index-$1$ cover, and let $C_i^\sharp:=(\pi^{-1}(C_i))_{\red}$. By \cite[A3]{Mori:flip}
there exist an embedding $X^\sharp\subset \CC^4$ and local coordinates $y_1,\dots,y_4$ near $P^\sharp$ such that $\wt_{\mumu_4}(y_1,\dots,y_4)=(1,3,3,2)$ and
for some semi-invariants $g_3,\, g_4$ with $\mult_0(g_2)\ge 1$ and $\mult_0(g_4)\ge 2$
we have
\begin{eqnarray*}
X^\sharp &=&\{y_1^2-y_2^2+g_3y_3+g_4y_4=0\},
\\
C^\sharp_1&=&\{y_1^2-y_2^2=y_3=y_4=0\}.
\end{eqnarray*}
Similarly, for $P\in C_2\subset X$ there exist local coordinates $z_1,\dots,z_4$
such that modulo permutations and changing the generator of $\mumu_4$ we have $\wt_{\mumu_4}(z_1,\dots,z_4)=(1,3,3,2)$ and
\begin{eqnarray*}
C^\sharp_2&=&\{z_1^2-z_4^3=z_2=z_3=0\}.
\end{eqnarray*}
Therefore, 
\[
\III_{C_2}^\sharp+(\mmm_P^\sharp)^2=(z_2,z_3)+(\mmm_P^\sharp)^2.
\]
Since $\mmm_P^\sharp=(y_1,\dots,y_4)=(z_1,\dots,z_4)$ and in view of the $\mumu_4$-action, we have
\[
\III_{C_1}^\sharp+\III_{C_2}^\sharp+(\mmm_P^\sharp)^2=(y_2,y_3, y_4)+(\mmm_P^\sharp)^2.
\]
Hence there exists a surjection
\[
\left(\upomega_{X^\sharp}\otimes \OOO / (\III_{C_1}^\sharp+\III_{C_2}^\sharp)\right)^{\mumu_4} \twoheadrightarrow
\left( \OOO_{X^\sharp}\cdot y_2 \otimes (\CC\oplus \CC\cdot y_1)\right)^{\mumu_4}=
\CC\cdot y_1y_2\neq 0.
\]
This contradicts \eqref{eq:ref(z:xic.3.6).2}.
\end{proof}
\end{scase}

\begin{scase}
\label{scase:IIB}
Thus we may assume that $C$ has no components of type \typec{II^\vee}.
Suppose that some component, say $C_1$ is of type~\typec{IIB}.
Then the germ $(X,C_1)$ must be divisorial \cite[Theorem~4.5]{KM92}
and any other component $C_j$, $j\neq 1$ is of type~\typec{IIA} by
\cite[1.4 (i)]{MP20}.
The restriction on the number of component follows from Lemma~\ref{lemma:components}.
We obtain the case~\ref{thm:main:IIB}. 
This finishes our treatment of \type{cAx/4}-points.
\end{scase}

\begin{scase}
\label{pf:thm:main:cD/3}
Suppose that $(X,P)$ is a point of type~\type{cD/3}.
Then by definition, each germ $(X,C_i)$ is birational of type~\typec{cD/3}.
The restriction on the number of component follows from Lemma~\ref{lemma:components}.
We obtain the case~\ref{thm:main:cD/3}. 
\end{scase}

Thus under the assumption~\ref{pf:1point}
it remains to consider the cases where all the components of $C$ are of types 
\typec{k1A}, \typec{cD/2}, \typec{cAx/2}, \typec{cE/2} or \typec{IC}.

\begin{scase} 
Suppose that $(X,P)$ is a point of type~\type{cAx/2}, \type{cD/2} or \type{cE/2}.
Since $C$ is reducible, by \cite[Theorem~4.7.4-6]{KM92} the germ $(X,C)$ is not birational. 
Thus $(X,C)$ is a $\QQ$-conic bundle germ.
Let $C_i\subset C$ be any component and let $C'=C-C_i$ be the remaining part.
Then $(X,C')$ is a birational curve germ \cite[Lemma~2.2]{MP20} and then $C'$ must be irreducible again by 
\cite[Theorem~4.7.4-6]{KM92}. We obtain the case~\ref{thm:main:cD/2-cE/2}. 
\end{scase}

Thus we may assume that~$P\in X$ is a point of type~\type{cA/n}.

\begin{scase}\label{pf:thm:main:IC}
Suppose that there is a component, say $C_1$, of type~\typec{IC}, then by \cite[Corollary~1.4(ii)]{MP20}
all other components $C_j\subset C$ are of type~\typec{k1A}.
The germ $(X,C_1)$ is flipping \cite[8.3.3]{KM92}. 
The restriction on the number of component follows from Lemma~\ref{lemma:components}.
We obtain the case~\ref{thm:main:IC}.

\end{scase}

Finally, if all the components $C_j\subset C$ are of type~\typec{k1A}, we obtain the case~\ref{thm:main:k1A}.

\subsection{}
From now on we assume that there is at least one component of $C$ 
containing two non-Gorenstein points of $X$. 
We note that if $C$ has a component $C_1$ of type~\typec{IC},
then it meets only components of types \typec{k1A} or \typec{k2A} by \cite[Corollary~1.4(ii)]{MP20},
and the case of \typec{k2A} is disproved by Theorem~\ref{thm:IC}. Hence $C_1$ can meet only \typec{k1A} components.
But then $(X,C)$ has a unique non-Gorenstein point. 
This contradicts our assumption and so $C$ has no components of type~\typec{IC}.

A component of type~\typec{k2A} 
contains only points of type \type{cA/n} and 
does not meet components of type~\typec{k3A}, \typec{kAD}, nor \typec{IC} by Theorems~\ref{thm:KAD-k3A} and~\ref{thm:IC}.
Thus the only components of $C$ that meet a \typec{k2A} component are of types \typec{k2A} or \typec{k1A}, and they exhaust $C$ if
$C$ contains component of type~\typec{k2A} since $C$ is connected. Then we obtain the case~\ref{thm:main:k2A+k1A}.
So from now on we assume that each component of $C$ containing two non-Gorenstein points is of type 
\typec{kAD} or \typec{k3A}. Let $C_1$ such a component.

\begin{scase} 
\label{pf:thm:main:k3A}
Suppose that $C_1$ is of type~\typec{k3A}. Then $(X,C_1)$ is divisorial by \cite[(5.1)]{KM92}, and
by \cite[Corollary~1.4(ii)]{MP20} and Theorem~\ref{thm:KAD-k3A}
all other components $C_j\subset C$ are of type~\typec{k1A}. 
The restriction on the number of component follows from Lemma~\ref{lemma:components}.
So, we obtain the case~\ref{thm:main:k3A}.
\end{scase}

\begin{scase} \label{pf:thm:main:kAD}
Finally, let $C_1$ is of type~\typec{kAD}.
Then by \cite[Corollary~1.4(iii)]{MP20} and Theorem~\ref{thm:KAD-k3A}
all other components $C_j\subset C$ are of type~\typec{k1A}, \typec{cD/2} or \typec{cAx/2}.
The restriction on the number of component follows from Lemma~\ref{lemma:components}.
We obtain the case~\ref{thm:main:kAD}. 
\end{scase}
This finishes the proof of Theorem~\ref{thm:main}.

\section{Examples}
\label{sect:ex}
In this section we collect some examples of extremal curve germs with reducible central curve.
We use the notation of Theorem~\ref{thm:main}.

\subsection{}
The case~\ref{thm:main:Gor} is trivial. Indeed, in this case 
the base $Z$ is smooth and
there is an embedding $X \hookrightarrow \PP^2\times Z$
so that the contraction
$f$ is induced by the projection to the second factor and the fibers are conics (see e.g. \cite{Cutkosky-1988}).

\subsection{}
Here is an example of extremal curve germs as in~\ref{thm:main:cD/2-cE/2} of Theorem~\ref{thm:main}.

\begin{sexample}[case~\ref{thm:main:cD/2-cE/2}, cf. {\cite[Example~3.12]{P97}}]
By \cite[Theorem 12.1]{MP:cb1} the threefold $X$ can be embedded to $\PP(1,1,1,2)\times Z$, where $Z\subset \CC^2$ is a small disk around the origin,
so that the contraction $f: X\to Z$ is the projection to the second factor and $C$ is the fiber over $(0,0)\in Z$.
Let $u,\, v$ be coordinates on $\CC^2$ and let $x_1,\, x_2,\, x_3,\, y$ be quasi-homogeneous 
coordinates of $\PP(1,1,1,2)$ with $\deg x_i=1$, $\deg y=2$. Suppose that $X$ is given by the equations
\begin{equation*}
x_1x_2- y u=
x_3^2+ y(c_1v^2+c_2v^3+c_3uv+c_4u^2v)+ u(x_1^2+x_2^2)=0,
\end{equation*}
where $c_1,\dots, c_4 \in \CC$. 
The central fiber $C$ of the projection $X\to\CC^2$ has two components 
meeting at~$P:=(0,0,0,1;0,0)$.
In the affine chart $y\neq 0$ we can eliminate $u$ and obtain a local equation of~$X$ at~$P$:
\[
\bigl\{x_3^2+c_1v^2+c_2v^3+c_3vx_1x_2+c_4x_1^2x_2^2v + x_1x_2(x_1^2+x_2^2)=0\bigr\}/\mumu_2(1,1,1,0).
\]
According to the classification of threefold terminal singularities (see e.g. \cite{Mori:sing} or \cite{Reid:YPG})
we have
\begin{itemize}
\item 
if $c_1 \neq 0 $, $c_2=c_3=c_4=0$, then $P \in X$ is of type~\type{cAx/2};
\item 
if $c_1 = c_4 = 0 $, $c_2 \neq 0 $, $c_3 \neq 0 $, then $P \in X$ is of type~\type{cD/2};
\item 
if $c_1 = c_3 = 0 $, $c_2 \neq 0 $,  $c_4 \neq 0 $, then $P \in X$ is of type~\type{cE/2}.
\end{itemize} 
\end{sexample}

\subsection{Construction}
\label{subs:def}
In many cases to show the existence of certain extremal curve germs one can use deformation arguments as follows. 
Suppose we are given a normal surface germ $(H, C)$ along a proper curve $C$ 
and a contraction $f_H : H \to H_Z$ such that $C$ is a fiber and such that $K_H$ is negative on the components of $C$. Let $o:= f_H(C)$.
The contraction $f_H$ must be either birational or a rational curve fibration.
Let $P_1,\dots,P_r \in H$ be singular points. Assume also that for a neighborhood $U_i\subset H$ of each point $P_i$ 
there is a small one-parameter deformation 
$V_i/\Delta$ over a small disc $0\in \Delta\subset \CC$ 
such that $U_{P_i}$ is the fiber over $0\in \Delta$ and
the total space $V_i$ has a terminal singularity at~$P_i$. 
On can show that the natural map
\[
\Def H \longrightarrow \coprod \Def(U_i, P_i)
\]
is smooth. 
Hence there is a global one-parameter deformation $X/\Delta$ of $H$ inducing a local deformation of $V_i/\Delta$ near each $P_i$
and such that $H$ is the fiber over $0\in \Delta$. 
This shows the existence of a threefold $X \supset H\supset C$ that has the desired structure near each $P_i$. The contraction $f : X \to Z$ exists by arguments similar to those in 
\cite[Theorem~6.3]{MP-1p} and \cite[11.4.1]{KM92}. 

\begin{sremark}
\label{rem:computH}
In the above notation the negativity of $K_H$ on a component $C_i\subset C$ can be checked as follows. 
Let $\mu: \tilde H\to H$ be the minimal resolution and let $\tilde C_i\subset \tilde H$ 
be the proper transform. Then we can write $K_{\tilde H}=\mu^* K_H-\Delta$, where $\Delta$ is the codiscrepancy divisor
which must be effective. Therefore, 
\begin{equation}
\label{eq:KH:C}
K_H\cdot C_i=K_{\tilde H}\cdot \tilde C_i+\Delta\cdot \tilde C_i\ge K_{\tilde H}\cdot \tilde C_i.
\end{equation} 
Thus $K_H\cdot C_i$ is negative if and only if $\tilde C_i$ is a $(-1)$-curve and $\Delta\cdot \tilde C_i<1$.
\end{sremark}

\begin{snotation}
In the examples below the surface $H$ has only rational singularities (cf. \cite[Theorem~1.8]{KM92}
and \cite[Lemma~2.4]{MP:IA}).
Hence its minimal resolution $\mu:\tilde H\to H$ can be described in terms of dual graphs. 
We will use the usual notation: each vertex labeled by $\bullet$ or $\scriptstyle{\blacksquare}$ corresponds to an irreducible component of $C$
which is a $(-1)$-curve, and each one labeled by $\circ$ corresponds to a component $E_i \subset \Supp(\Delta)$ 
which is a smooth rational curve. 
A number attached to a vertex denotes minus the self-intersection number. For short, we omit $2$ when this number equals~$-2$.
\end{snotation}

\begin{example}[case {\ref{thm:main:cD/3}}]
\label{ex:cd/3}
Consider a smooth surface $\tilde H$ containing a configuration of smooth rational curves whose dual graph is one of
the following:
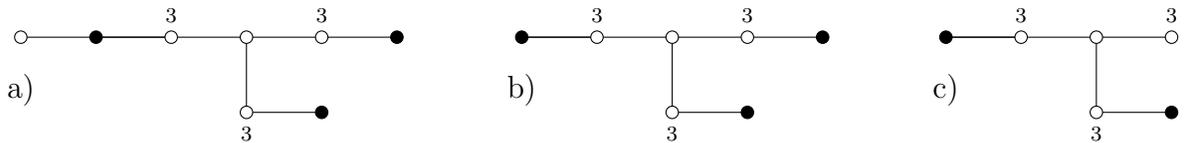
\def\sizec{0.2em}
\[
\begin{tikzpicture}
\draw (1,0.3) node {a)};
\coordinate (1) at (1,1);
\coordinate (2) at (2,1);
\coordinate (3) at (3,1);
\coordinate (4) at (4,1);
\coordinate (5) at (5,1);
\coordinate (6) at (6,1);
\coordinate (7) at (4,0);
\coordinate (8) at (5,0);

	\path (1) edge (2);
	\path (2) edge (3);
	\path (2) edge (3);
	\path (3) edge (4);
	\path (4) edge (5);
	\path (5) edge (6);
	\path (4) edge (7);
	\path (7) edge (8);

\path[fill=white,draw=black] (1) circle (\sizec);
\path[fill=black,draw=black] (2) circle (\sizec);
\path[fill=white,draw=black] (3) circle (\sizec) node [above, yshift=1.5] {$\scriptstyle 3$};
\path[fill=white,draw=black] (4) circle (\sizec);
\path[fill=white,draw=black] (5) circle (\sizec) node [above, yshift=1.5] {$\scriptstyle 3$};
\path[fill=black,draw=black] (6) circle (\sizec);
\path[fill=white,draw=black] (7) circle (\sizec) node [below, yshift=-1.5] {$\scriptstyle 3$};
\path[fill=black, draw=black] (8) circle (\sizec);

\end{tikzpicture}
\hspace{3em}
\begin{tikzpicture}
\draw (2,0.3) node {b)};

\coordinate (2) at (2,1);
\coordinate (3) at (3,1);
\coordinate (4) at (4,1);
\coordinate (5) at (5,1);
\coordinate (6) at (6,1);
\coordinate (7) at (4,0);
\coordinate (8) at (5,0);

	\path (2) edge (3);
	\path (2) edge (3);
	\path (3) edge (4);
	\path (4) edge (5);
	\path (5) edge (6);
	\path (4) edge (7);
	\path (7) edge (8);

\path[fill=black,draw=black] (2) circle (\sizec);
\path[fill=white,draw=black] (3) circle (\sizec) node [above, yshift=1.5] {$\scriptstyle 3$};
\path[fill=white,draw=black] (4) circle (\sizec);
\path[fill=white,draw=black] (5) circle (\sizec) node [above, yshift=1.5] {$\scriptstyle 3$};
\path[fill=black,draw=black] (6) circle (\sizec);
\path[fill=white,draw=black] (7) circle (\sizec) node [below, yshift=-1.5] {$\scriptstyle 3$};
\path[fill=black, draw=black] (8) circle (\sizec);

\end{tikzpicture}
\hspace{3em}
\begin{tikzpicture}
\draw (2,0.3) node {c)};

\coordinate (2) at (2,1);
\coordinate (3) at (3,1);
\coordinate (4) at (4,1);
\coordinate (5) at (5,1);

\coordinate (7) at (4,0);
\coordinate (8) at (5,0);

	\path (2) edge (3);
	\path (2) edge (3);
	\path (3) edge (4);
	\path (4) edge (5);

	\path (4) edge (7);
	\path (7) edge (8);

\path[fill=black,draw=black] (2) circle (\sizec);
\path[fill=white,draw=black] (3) circle (\sizec) node [above, yshift=1.5] {$\scriptstyle 3$};
\path[fill=white,draw=black] (4) circle (\sizec);
\path[fill=white,draw=black] (5) circle (\sizec) node [above, yshift=1.5] {$\scriptstyle 3$};

\path[fill=white,draw=black] (7) circle (\sizec) node [below, yshift=-1.5] {$\scriptstyle 3$};
\path[fill=black, draw=black] (8) circle (\sizec);

\end{tikzpicture}
\]
The configuration can be contracted in two stages:
\[
\varphi: \tilde H \overset{\mu}\longrightarrow H \overset{f_H}\longrightarrow H_Z 
\]
where $\mu$ is the (birational) contraction of the curves corresponding to white vertices 
and $f_H$ is the contraction of the remaining curves $C_i\subset H$ corresponding to the $\bullet$-vertices.
The surface $H$ has a unique singular point $P$, which is log canonical; the curves $C_i$ are smooth rational and pass through $P$.
Here $\mu$ is the minimal resolution of $P$. 
In the notation of Remark~\ref{rem:computH} we have $\Delta\cdot \tilde C_i=2/3<1$, hence $K_H$ is negative on~$C_i$. 
According to \cite[6.7]{KM92} a neighborhood $P\in U \subset H$ can be realized as a hyperplane section of a terminal 
\type{cD/3}-singularity. By the arguments of~\ref{subs:def} this gives an example of an extremal curve germ
of type
\ref{thm:main:cD/3}.
\end{example}

\begin{sclaim}
The extremal curve germ $(X,C)$ constructed in  Example~\xref{ex:cd/3}  is a $\QQ$-conic bundle with $N=3$ in the case~\textup{a)}, divisorial with $N=3$ in the case~\textup{b)}, and flipping with 
$N=2$ in the case~\textup{c)}.
\end{sclaim}

\begin{proof}
Indeed, in the case~a) the self-intersection of the central curve $C\subset H$ equals $0$.
Hence $f_H : H\to H_Z$ is a fibration and so $f$~is. In the case~b) the contraction $f_H$ is birational and the image $H_Z$ has  a Du Val \type{D_4}-singularity at $o=f_H(C)$. Hence $Z$ is a 
Gorenstein threefold and $f$ is divisorial.

Finally, consider the case~c), the contraction $f_H$ is birational and the point $H_Z \ni o$ is log
terminal but it is not Du Val neither a cyclic quotient. Therefore, it is not of type~\type{T} \cite[Proposition~3.10]{KSh88} and so the singularity $Z \ni o$ is not $\QQ$-Gorenstein. Let $C_1, C_2$ 
be components of
$C$. Then $C_1 \cap C_2$ is a \type{cD/3}-point of $X$. Run the analytic MMP on $X$ over $Z$:
$$
X \overset{f_1}{\dasharrow} X' \overset{f_2}{\dasharrow} X'' \dasharrow \cdots .
$$
Here each step $X^{i} \dasharrow X^{i+1}$ is either an analytic flip or divisorial contraction with irreducible
central fiber. 

If $f_1 : X \to X'$ is a divisorial contraction, then $X'$ is Gorenstein \cite[Theorem~1.5]{MP:IA}
and $C' := f_1(C)$ is irreducible,
in which case, the dual graph of~c) does not quite fit with those in \cite[Theorem~1.5]{MP:IA}. 
Therefore, the first step $X \dasharrow X'$ is a
flip and the central fiber $C' \subset X'$ over o has two components: $C' = C'_1 \cup C'_2$. 
We may assume that $C'_1$ is the flipped curve. Thus $\operatorname{index}(X')\le 2$ and 
$K_{X'} \cdot C'_1=1/\operatorname{index}(X')\le 1$ by Table~\ref{table2} and $K_{X'} 
\cdot C'_2 < 0$ by Lemma~\ref{lemma:flip}.

If the second step of the MMP is a divisorial contraction $f_2 : X' \to X''$, which necessarily contracts $C'_2$, then $X''$ must be
Gorenstein \cite[Theorem~4.7]{KM92}. As above $C'' := f_2(C'_1)$ is irreducible, $K_{X''} \cdot C'' <1$ by Lemma \ref{lemma:div}. Since $X''$ is Gorenstein, this implies that $K_{X''} \cdot C'' \le 0$ 
and $Z$ is also Gorenstein, while we saw earlier that $Z$ is not. Thus we get 
a contradiction. Therefore, $f_2 : X' \dasharrow X''$ is a flip.

In this case, $X'$ is has a (unique) point of index~$2$
and $K_{X'} \cdot C'_1 = 1/2$ (see Table~\ref{table2}). Then $X''$ is Gorenstein 
\cite[Theorem 4.2]{KM92} and the central
fiber $C'' \subset X''$ over $o$ has two components. For the $f_2$-flipped curve $C''_2$ we have $K_{X''} \cdot C''_2 = 1$ (see Lemma~\ref{lemma:flips2-2}). For another component $C'_1\subset C''$ we 
have 
$K_{X''} \cdot C''_1<1/2$  by Lemma~\ref{lemma:flip}. Since $X''$ is Gorenstein, we have $K_{X''} \cdot C''_1\le 0$.
If $K_{X''} \cdot C''_1=0$, then $X''$  is a minimal model over $Z$, hence the contraction $X''\to Z$ is small and our original contraction $X\to Z$ is flipping. Assume that $K_{X''} \cdot C''_1<0$, 
then 
the next step of the MMP is a divisorial contraction $f_3: X'' \to X'''$ so that $f_3(C''_1)$ is a point, the central fiber $C''' := f_3(C''_2)$ is irreducible, and $X'''$ is Gorenstein.
Moreover, we have $K_{X'''} \cdot C''' <K_{X''} \cdot C''_2=1$. Hence $K_{X'''} \cdot C'''\le 0$.
As above we obtain that the $Z$ is also Gorenstein, while we saw earlier that $Z$ is not.
\end{proof}

\begin{example}[case~\ref{thm:main:IIA}]
As above, consider the following configuration
\[
\def\sizec{0.2em}
\def\sl{0.7}
\begin{tikzpicture}[square/.style={regular polygon,regular polygon sides=4}]
\def\sizec{0.2em}
\coordinate (1) at (0,0);
\coordinate (2) at (1,0);
\coordinate (3) at (2,0);
\coordinate (4) at (3,0);
\coordinate (5) at (4,0);
\coordinate (6) at (1,1*\sl);
\coordinate (7) at (2,1*\sl);
\coordinate (8) at (3,1*\sl);
\coordinate (9) at (1,-1*\sl);

	\path (1) edge (2);
	\path (2) edge (3);
	\path (2) edge (3);
	\path (3) edge (4);
	\path (4) edge (5);
	\path (6) edge (2);
	\path (3) edge (7);
	\path (8) edge (4);
	\path (9) edge (2);
\node at (1) {$\bullet$};
\path[fill=white,draw=black] (2) circle (\sizec) node [below, yshift=-2,xshift=-5] {$\scriptstyle 4$};
\path[fill=white,draw=black] (3) circle (\sizec);
\path[fill=white,draw=black] (4) circle (\sizec) node [below, yshift=-3] {$\scriptstyle 4$};
\path[fill=black,draw=black] (5) circle (\sizec) ;
\path[fill=black,draw=black] (6) circle (\sizec);
\path[fill=white,draw=black] (7) circle (\sizec);
\path[fill=black, draw=black] (8) circle (\sizec);
\node at (9) {$\bullet$};
\end{tikzpicture}
\]
Using the local computations from \cite[7.7.1]{KM92} we can construct a $\QQ$-conic bundle germ $(X,C)$ with $N=5$ whose components are 
of type \typec{IIA} (as in~\ref{thm:main:IIA}). Removing some $\bullet$-vertices we obtain 
divisorial germs with $N=4$ and flipping ones with $N\le 3$.
\end{example}

\subsection{}
Extremal curve germs as in~\ref{thm:main:IIdual2} were described by explicit equations in \cite[(1.3.2)]{MP:cb2}. 
Note however that in this case $(X, C)$ is a $\QQ$-conic bundle germ 
over a surface having an \type{A_1}-singularity.
Then any hyperplane section $H\in |\OOO_X|_C$ is analytically reducible
and our construction~\ref{subs:def} does not work.

\begin{example}[case~\ref{thm:main:IIdual}]
Consider a smooth surface $\tilde H$ containing a configuration of smooth rational curves with 
the following dual graph
\[
\begin{tikzpicture}[square/.style={regular polygon,regular polygon sides=4}]
\def\sizec{0.2em}
\def\sl{0.7}
\coordinate (1) at (0,0);
\coordinate (2) at (1,0);
\coordinate (3) at (2,0);
\coordinate (4) at (3,0);
\coordinate (5) at (4,0);
\coordinate (6) at (1,1*\sl);
\coordinate (7) at (2,1*\sl);
\coordinate (8) at (3,1*\sl);
\coordinate (9) at (2,2*\sl);

	\path (1) edge (2);
	\path (2) edge (3);
	\path (2) edge (3);
	\path (3) edge (4);
	\path (4) edge (5);
	\path (6) edge (2);
	\path (3) edge (7);
	\path (8) edge (4);
	\path (9) edge (7);
\path[fill=black,draw=black] (1) circle (\sizec);
\path[fill=white,draw=black] (2) circle (\sizec) node [below, yshift=-1.5] {$\scriptstyle 4$};
\path[fill=white,draw=black] (3) circle (\sizec);
\path[fill=white,draw=black] (4) circle (\sizec) node [below, yshift=-1.5] {$\scriptstyle 4$};
\path[fill=black,draw=black] (5) circle (\sizec) ;
\path[fill=black,draw=black] (6) circle (\sizec);
\path[fill=white,draw=black] (7) circle (\sizec);
\path[fill=black, draw=black] (8) circle (\sizec);
\node at (9) {$\scriptstyle{\blacksquare}$};
\end{tikzpicture}
\]
According to \cite[7.7.1]{KM92} a neighborhood $P\in U \subset H$ can be realized as a hyperplane section of a terminal 
\type{cAx/4}-singularity. 
This gives us an example of a $\QQ$-conic bundle extremal curve germ whose central fiber has $5$ components
$C_{0},\dots,C_4$ meeting at the \type{cAx/4}-point $P$.
Let $E_{0},\dots,E_{3}$ be the $\mu$-exceptional curves on $\tilde H$ numbered so that
$E_{0}$ corresponds to the central vertex,
$E_1$ and $E_2$ correspond to the $\overset{4}\circ$-vertices, 
and 
$E_{3}$ corresponds to the remaining $\circ$-vertex. 
Then for the codiscrepancy of $P\in H$ (see Remark~\ref{rem:computH}) we have 
$\Delta=E_{0}+\frac 34 E_1+ \frac 34 E_2+ \frac 12 E_{3}$. 
We may assume that $C_1,\dots,C_4$ correspond to $\bullet$-vertices
and $C_{0}$ corresponds to the $\scriptstyle{\blacksquare}$-vertex.
Then 
$K_X\cdot C_i=\Delta\cdot \tilde C_i-1=-1/4$ for $i=1,\dots,4$ and $K_X\cdot C_{0}=-1/2$ (see~\eqref{eq:KH:C}).
Since $K_X$ is a generator of the group $\Clsc(X,P)\simeq \ZZ/4\ZZ$,
the natural map $\Clsc(X,P)\to \frac14 \ZZ/\ZZ$ given by local intersections with $C_i$ 
is surjective for $i=1,\dots,4$ and it is not surjective for $i=0$
(see \eqref{eq:def:imp}). 
Hence $C_1,\dots,C_4$ are locally primitive and $C_{0}$ locally imprimitive at~$P$. 
Removing some $\bullet$-components we obtain 
divisorial germs with $N=3$ or $4$ and flipping ones with $N=2$.
\end{example}

\begin{example}[case~\ref{thm:main:IIB}]
Consider a smooth surface $\tilde H$ containing a configuration of smooth rational curves with 
the following dual graph
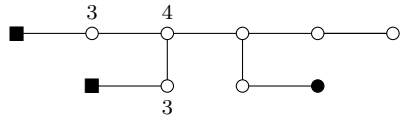
\[
\begin{tikzpicture}
\def\sizec{0.2em}
\def\sl{0.7}
\coordinate (1) at (0,0);
\coordinate (2) at (1,0);
\coordinate (3) at (2,0);
\coordinate (4) at (3,0);
\coordinate (5) at (4,0);
\coordinate (6) at (5,0);
\coordinate (7) at (2,-1*\sl);
\coordinate (8) at (3,-1*\sl);
\coordinate (9) at (1,-1*\sl);
\coordinate (10) at (4,-1*\sl);
\draw (1) -- (2);
\draw (2) -- (3);
\draw (3) -- (4);
\draw (4) -- (5);
\draw (5) -- (6);
\draw (3) -- (7);
\draw (4) -- (8);
\draw (5) -- (6);
\draw (7) -- (9);
\draw (8) -- (10);
\node at (1) {$\scriptstyle{\blacksquare}$};
\path[fill=white,draw=black] (2) circle (\sizec) node [above, yshift=1.5] {$\scriptstyle 3$};
\path[fill=white,draw=black] (3) circle (\sizec) node [above, yshift=1.5] {$\scriptstyle 4$};
\path[fill=white,draw=black] (4) circle (\sizec);
\path[fill=white,draw=black] (5) circle (\sizec);
\path[fill=white,draw=black] (6) circle (\sizec);

\path[fill=white,draw=black] (7) circle (\sizec) node [below, yshift=-1.5] {$\scriptstyle 3$};
\path[fill=white,draw=black] (8) circle (\sizec);
\node at (9) {$\scriptstyle{\blacksquare}$};
\path[fill=black,draw=black] (10) circle (\sizec);
\end{tikzpicture}
\]
According to \cite[\S~3]{MP:ICIIB} a neighborhood $P\in U \subset H$ can be realized as a hyperplane section of a terminal 
\type{cAx/4}-singularity $P\in V=\{y_1^2-y_2^3+y_3y_4+\phi(y_1,\dots,y_4)=0\}/\mumu_4(3,2,1,1)$. 
Moreover, the pull-back to the index-$1$ cover of $V$ of the curve $C_1$ corresponding to the $\bullet$-vertex 
is singular. 
This means that $(V\supset C_1)$ is of type \typec{IIB} at~$P$ \cite[4.2]{Mori:flip}.
Hence by the arguments of~\ref{subs:def} this gives an example of an extremal curve germ $(X,C)$
having a type~\typec{IIB} component $C_1$. By \cite[1.4 (i)]{MP20} and
\ref{case:IIvee} any other component $C_j\subset C$ is of type~\typec{IIA}.
Thus $(X,C)$ as in~\ref{thm:main:IIB}. One can see that $(X,C)$ is a conic bundle germ with $N=3$. Removing one component 
corresponding to a $\scriptstyle{\blacksquare}$-vertex we obtain a divisorial germ with $N=2$.
\end{example}

\begin{example}[case~\ref{thm:main:IC}]
As above, consider the following configuration
\[
\begin{tikzpicture}
\def\sizec{0.2em}
\def\sl{0.7}
\coordinate (1) at (0,0);
\coordinate (2) at (1,0);
\coordinate (3) at (-1,1*\sl);
\coordinate (4) at (0,1*\sl);
\coordinate (5) at (1,1*\sl);
\coordinate (6) at (2,1*\sl);
\coordinate (7) at (0,2*\sl);
\coordinate (8) at (1,2*\sl);
\coordinate (9) at (2,2*\sl);
\coordinate (10) at (1,3*\sl);
\draw (1) -- (4);
\draw (2) -- (5);
\draw (3) -- (4);
\draw (4) -- (5);
\draw (5) -- (6);
\draw (5) -- (8);
\draw (8) -- (9);
\draw (8) -- (10);
\draw (8) -- (7);
\path[fill=black,draw=black] (1) circle (\sizec);
\path[fill=white,draw=black] (2) circle (\sizec);
\path[fill=white,draw=black] (3) circle (\sizec);
\path[fill=white,draw=black] (4) circle (\sizec);
\path[fill=white,draw=black] (5) circle (\sizec) node [above, yshift=1.5, xshift=4.5 ] {$\scriptstyle 4$};
\path[fill=white,draw=black] (6) circle (\sizec);
\node at (7) {$\scriptstyle{\blacksquare}$};
\path[fill=white,draw=black] (8) circle (\sizec) node [above, yshift=1.5, xshift=4.5 ] {$\scriptstyle 4$};
\node at (9) {$\scriptstyle{\blacksquare}$};
\node at (10) {$\scriptstyle{\blacksquare}$};
\end{tikzpicture}
\]
Using \cite[10.7.3.1]{KM92} we construct a $\QQ$-conic bundle germ $(X,C)$ with $N=4$ having a component $C_1\subset C$ 
of type \typec{IC} of index $5$ (as in~\ref{thm:main:IC}).
By \cite[Corollary~1.4(ii)]{MP20}
all other components $C_j\subset C$ are of type~\typec{k1A}.
Removing some $\scriptstyle{\blacksquare}$-components we obtain 
a divisorial germ with $N=3$ and a flipping one with $N=2$.
\end{example}

\begin{example}[case~\ref{thm:main:k1A}]
\label{ex:case:k1A-a}
Consider the following configuration
\[
\begin{tikzpicture}
\def\sizec{0.2em}
\def\sl{1.1}
\coordinate (1) at (0,0);
\coordinate (2) at (1,0);
\coordinate (3) at (2,0);
\coordinate (4) at (3,0);
\coordinate (5) at (4,0);
\coordinate (6) at (5,0);
\coordinate (7) at (7,1*\sl);
\coordinate (8) at (7,-1*\sl);
\coordinate (9) at (7,0*\sl);
\coordinate (10) at (7,0.5*\sl);
\coordinate (11) at (7,-0.5*\sl);
\draw (1) -- (2);
\draw (2) -- (3);
\draw [loosely dotted] (3) -- (4);\draw (4) -- (5);\draw (5) -- (6);
\draw (6) -- (7);\draw (6) -- (8); \draw (6) -- (9);

\path[fill=black,draw=black] (1) circle (\sizec);
\path[fill=white,draw=black] (2) circle (\sizec);

\path[fill=white,draw=black] (5) circle (\sizec);
\path[fill=white,draw=black] (6) circle (\sizec) node [below, yshift=-1.5, xshift=-4.5 ] {$\scriptstyle {m+2}$};
\path[fill=black,draw=black] (7) circle (\sizec);
\path[fill=black,draw=black] (8) circle (\sizec);
\path[fill=black,draw=black] (9) circle (\sizec);
\node at (10) {$\vdots$};
\node at (11) {$\vdots$};
\draw [thick, black,decorate,decoration={brace,amplitude=10pt,mirror, raise=5pt},xshift=0.4pt,yshift=-1.4pt](2) -- (5) node[black,midway,yshift=-23] {\footnotesize $m-2$};
\draw [thick, black,decorate,decoration={brace,amplitude=10pt, raise=5pt},xshift=0.4pt,yshift=-1.4pt] (7) -- (8) node[black,midway,xshift=23] {\footnotesize $N$};
\end{tikzpicture}
\]
where $m\ge3$ and $2\le N\le m+1$.
Then in the above notation, the surface $H$ has a unique singular point $P$ which is a cyclic quotient of type 
$\frac{1}{m^2}\big(1,\, m(m-1)-1\big)$.
This is a singularity of type \type{T} \cite[Proposition~3.10]{KSh88}. 
This means that $(H,P)$ has a one-parameter deformation with terminal
total space. By~\ref{subs:def} we obtain an example of extremal curve germ whose central fiber has $N$ components.
The germ is a $\QQ$-conic bundle if $N=m+1$, it is divisorial if $N= m$ or $m-1$, and it is flipping if $N<m-1$.
\end{example}

The following example shows that in the case~\ref{thm:main:k1A} some of the components of $C$ can be locally imprimitive.

\begin{example}[case~\ref{thm:main:k1A}]
\label{ex:case:k1A-c}
Let $k$ be an integer $\ge 2$ and let $m:=2k(2k-1)$. 
Consider the following configuration
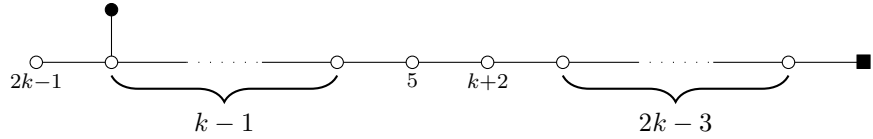
\[
\begin{tikzpicture}
\def\sizec{0.2em}
\def\sl{0.7}
\coordinate (Q2) at (1,1*\sl);
\coordinate (P1) at (0,0);
\coordinate (P2) at (1,0);
\coordinate (P3) at (2,0);
\coordinate (P4) at (3,0);
\coordinate (P5) at (4,0);
\coordinate (P6) at (5,0);
\coordinate (P7) at (6,0);
\coordinate (P8) at (7,0);
\coordinate (P9) at (8,0);
\coordinate (P10) at (9,0);
\coordinate (P11) at (10,0);
\coordinate (Q12) at (11,0);

\draw (P1) -- (P2);
\draw (P2) -- (P3);
\draw [loosely dotted] (P3) -- (P4);
\draw (P4) -- (P5);
\draw (P5) -- (P6);
\draw (P6) -- (P7);
\draw (P7) -- (P8);
\draw (P8) -- (P9);
\draw [loosely dotted] (P9) -- (P10);
\draw (P10) -- (P11);
\draw (P2) -- (Q2);
\draw (P11) -- (Q12);

\path[fill=white,draw=black] (P1) circle (\sizec)node [below, yshift=0, xshift=0] {$\scriptstyle 2k-1$};
\path[fill=white,draw=black] (P2) circle (\sizec)node [below, yshift=0, xshift=0] {};
\path[fill=white,draw=black] (P5) circle (\sizec)node [below, yshift=0, xshift=0] {};
\path[fill=white,draw=black] (P6) circle (\sizec)node [below, yshift=0, xshift=0] {$\scriptstyle 5$};\path[fill=white,draw=black] (P7) circle (\sizec)node [below, yshift=0, xshift=0] {$\scriptstyle 
k+2$};\path[fill=white,draw=black] (P8) circle (\sizec)node [below, yshift=0, xshift=0] {};
\path[fill=white,draw=black] (P11) circle (\sizec)node [below, yshift=0, xshift=0] {};
\path[fill=black,draw=black] (Q2) circle (\sizec)node [below, yshift=0, xshift=0] {};
\node at (Q12) {$\scriptstyle{\blacksquare}$};

\draw [thick, black,decorate,decoration={brace,amplitude=10pt,mirror, raise=5pt},xshift=0.4pt,yshift=-1.4pt](P2) -- (P5) node[black,midway,yshift=-23] {\footnotesize $k-1$};
\draw [thick, black,decorate,decoration={brace,amplitude=10pt,mirror, raise=5pt},xshift=0.4pt,yshift=-1.4pt](P8) -- (P11) node[black,midway,yshift=-23] {\footnotesize $2k-3$};
\end{tikzpicture}
\]
As above the chain of white vertices is contracted to a cyclic quotient singularity $P\in H$ 
of type $\frac1{m^2}(1,\, m(2k+1)-1)$ which is a singularity $H\ni P$ of type \type{T} and index $m$ \cite[Proposition~3.10]{KSh88}. 
Thus this $H\ni P$ can be realized as a hyperplane section of a
terminal cyclic quotient singularity of type $\frac1{m}(1,\,-1,\, 2k+1)$.
This gives us an example of a flipping extremal curve germ whose central fiber has $2$ components
meeting at the point $P$ of index $m$.
Let $E_1,\dots,E_{3k-1}$ be the curves on $\tilde H$ corresponding to the white vertices numbered from the 
left to the right. One can compute that in the notation of Remark~\ref{rem:computH} the coefficient of $E_2$ 
in $\Delta$ equals $(2k-1)/(2k)$. Hence, for the component $C_1\subset C$ corresponding to 
$\bullet$-vertex we have $K_X\cdot C_1=\Delta\cdot \tilde C_1-1=-1/(2k)$.
This implies that the natural map $\ZZ/m\ZZ\simeq \Clsc(X,P)\to \frac1m \ZZ/\ZZ$ given by 
the local intersection with $C_1$ is not surjective and, moreover, 
the component $C_1$ is locally imprimitive of splitting degree $2k-1$ (see~\eqref{eq:def:imp}). Another component, 
that corresponds $\scriptstyle\blacksquare$-vertex, is primitive.
\end{example}

Unfortunately, we don't know any examples of extremal curve germs as in~\ref{thm:main:k3A}.
The method outlined in~\ref{subs:def} is not applicable here because a general hyperplane section 
$H\in |\OOO_X|_C$ cannot be normal in this case.

\begin{example}[case~\ref{thm:main:kAD}]
Consider the following configuration
\[
\begin{tikzpicture}
\def\sizec{0.2em}
\def\sl{0.7}
\coordinate (11) at (-3,1*\sl);
\coordinate (1) at (-2,1*\sl);
\coordinate (2) at (1,0);
\coordinate (3) at (-1,1*\sl);
\coordinate (4) at (0,1*\sl);
\coordinate (5) at (1,1*\sl);
\coordinate (6) at (2,1*\sl);
\coordinate (7) at (0,2*\sl);
\coordinate (8) at (1,2*\sl);
\coordinate (9) at (2,2*\sl);
\coordinate (10) at (1,3*\sl);
\draw (1) -- (3);
\draw (1) -- (11);
\draw (2) -- (5);
\draw (3) -- (4);
\draw (4) -- (5);
\draw (5) -- (6);
\draw (5) -- (8);
\draw (8) -- (9);
\draw (8) -- (10);
\draw (8) -- (7);
\path[fill=black,draw=black] (1) circle (\sizec);
\path[fill=white,draw=black] (2) circle (\sizec);
\path[fill=white,draw=black] (3) circle (\sizec);
\path[fill=white,draw=black] (4) circle (\sizec);
\path[fill=white,draw=black] (5) circle (\sizec) node [above, yshift=1.5, xshift=4.5 ] {$\scriptstyle 4$};
\path[fill=white,draw=black] (6) circle (\sizec);
\node at (7) {$\scriptstyle{\blacksquare}$};
\path[fill=white,draw=black] (8) circle (\sizec) node [above, yshift=1.5, xshift=4.5 ] {$\scriptstyle 4$};
\node at (9) {$\scriptstyle{\blacksquare}$};
\node at (10) {$\scriptstyle{\blacksquare}$};
\path[fill=white,draw=black] (11) circle (\sizec) node [below, yshift=-1.5] {$\scriptstyle 4$};
\end{tikzpicture}
\]
Using \cite[10.7.3.1]{KM92} we construct a $\QQ$-conic bundle germ $(X,C)$ with $N=4$ having a component $C_1\subset C$ 
of type \typec{kAD} (as in~\ref{thm:main:kAD}). 
The $\underset{4}\circ$-vertex on the left hand side 
corresponds to a cyclic quotient singularity. 
By \cite[Corollary~1.4(iii)]{MP20}
all other components $C_j\subset C$ are of type~\typec{k1A}.
Removing some $\scriptstyle{\blacksquare}$-components we obtain 
divisorial germs with $N=3$ and flipping ones with $N=2$.
\end{example}

\appendix 
\section{General facts on extremal curve germs}
\label{sect:app}

\begin{definition}[{\cite[1.7]{Mori:flip}}]
\label{def:imp}
Let $(X, P )$ be a terminal three-dimensional singularity of index $m$ and let $C\subset X$ 
be a smooth curve passing through $P$. We say that $C$ is (locally) \textit{primitive} at~$P$ if the natural map 
\begin{equation}
\label{eq:def:imp}
\varrho: \Clsc(X,P) \longrightarrow \QQ/\ZZ
\end{equation} 
given by intersection with $C$ is injective. 
Otherwise $C$ is said to be (locally) \textit{imprimitive} at~$P$. 
The order of $\ker(\varrho)$ is called the (local) \textit{splitting degree} of $C$ at~$P$.
\end{definition}

The splitting degree is equal to the number of analytic components of the pull-back of $C$ 
under the index-$1$ cover \cite[1.7]{Mori:flip}.

\subsection{}
\label{not:app}
Let $(X,C)$ be the germ of a threefold with terminal singularities along a proper connected curve.
Assume that there exists a boundary $\Delta$ on $X$ such that the pair $(X,\Delta)$ is terminal.
Assume also that there exists a contraction 
\[
f: (X,C) \longrightarrow(Z,o)
\]
such that $f^{-1}(o)_{\red}=C$ and $-(K_X+\Delta)$ is $f$-ample.
Note that the boundary $\Delta$ as above automatically exists in the case where $(X,C)$ is an extremal curve
or flipped germ.
Similar to \cite[Corollary~1.3]{Mori:flip} we see that $C$ is a union of smooth rational curves 
such that $\p(C)=0$, $\Pic(X)\simeq \ZZ^\uprho(X)$, where $\uprho(X)$ is the number of components of $C$, and 
there is the following exact sequence (see \cite[(1.8.1)]{Mori:flip} and \eqref{eq:Term-sing}):
\begin{equation}
\label{exact-Clcs}
\vcenter{
\xymatrix@R=-4pt{
0\ar[r] & \Pic(X)\ar[r] & \Clsc (X)\ar[r]& \bigoplus \Clsc(X,P_i)\ar[r] & 0
\\
& \rotatebox{90}{$\simeq$}& &\rotatebox{90}{$\simeq$}&
\\
& \ZZ^{\uprho(X)}&& \bigoplus \ZZ/m_i\ZZ
}}
\end{equation}
where the sum runs through a finite number of singular points.
Hence, the group $\Clsc(X)$ is finitely generated and the group $\Cl(X)_{\tors}$ is finite.
Moreover, 
\begin{equation}
\label{eq:Cl=Cl}
\Cl(Z,o)_{\tors}\simeq \Cl(X)_{\tors}\subset \Oplus \Clsc(X,P_i).
\end{equation} 

\begin{sdefinition}
In the above notation we say that $(X,C)$ is (globally) \textit{primitive} if the group $\Clsc(X)$
is torsion free. Otherwise $(X,C)$ is said to be \textit{imprimitive}. The order of $\Cl(X)_{\tors}$
is called the \textit{splitting degree} of $(X,C)$.
\end{sdefinition}

\begin{lemma}
\label{lemma:cov}
Let $(X,C)$ be an extremal curve germ and let $\tau: (X',C')\to (X,C)$
a finite \'etale-in-codimension-two
cover, where $C'$ is connected.
Then $\tau$ is a cyclic cover and the preimage $\tau^{-1}(P)$ is a point for some $P\in C$.
Moreover, the degree of $\tau$ divides the index of $P\in X$.
\end{lemma}

\begin{proof}
We keep the notation of~\ref{not:app}.
We may assume that 
the cover $\tau$ is Galois with $G:=\operatorname{Gal}(X'/X)$.
By taking the Stein factorization we obtain the following commutative diagram
\begin{equation*}
\vcenter{
\xymatrix@R=1.5em{
(X',C')\ar[r]^{\tau}\ar[d]^{f'}& (X,C)\ar[d]^f
\\
(Z',o')\ar[r]^{\theta}&(Z,o)
}}
\end{equation*}
where $f'$ is an extremal curve germ.
By the construction the group $G$ 
acts on $(X',C')$ so that this action is free on the open set $X'\setminus \tau^{-1}(\Sing(X))$.
We will show that $G$ is cyclic.
Since $C^\prime=\cup C_i^\prime$ is a tree of smooth 
rational curves, it is easy to prove by induction 
on the number of components of $C^\prime$ that the group $G$ has either an
invariant component $C_i^\prime\subset C^\prime$ or a fixed point $P^\prime\in \Sing(C^\prime)$.
In the latter case, let $P=\tau(P^\prime)$. There is a surjection $\uppi_1(U\setminus
\{P\})\twoheadrightarrow G$, where $U\subset X$ is a small neighborhood of~$P$. Since $\uppi_1(U\setminus
\{P\})$ is cyclic (see~\eqref{eq:Term-sing}), we are done.

In the former case, let $C_i:=\tau (C^\prime_i)$.
Then both $(X,C_i)$ and $(X',C_i')$ are extremal curve germs, and the group $G$ acts on $(X',C_i')$
so that $(X',C_i')/G=(X,C_i)$.
We may replace $(X^\prime, C^\prime)$ with $(X^\prime, C^\prime_i)$ 
and $(X, C)$ with $(X, C_i)$. Thus
$C=C_i$, $C^\prime=C^\prime_i$, and $C^\prime/G=C\simeq \PP^1$. 
The branch points $P_1,\dots, P_n\in C$ of the cover $C^\prime\to C$ are non-Gorenstein points of $X$.
Assume that the group $G$ is not cyclic. Then $n>2$, hence $X$ has at least three non-Gorenstein points.
This contradicts \cite[Theorem~6.7 and~(10.6)]{Mori:flip} and \cite[Theorems~1.2 and~1.3]{MP:cb1}.
\end{proof}

\begin{scorollary}
\label{cor:Clsc}
Let $(X,C)$ be either an extremal curve germ or a flipped germ. Then $\Cl(X)_{\tors}$ is a cyclic group.
Moreover, there exists a point $P\in X$ such that the natural homomorphism
\begin{equation}
\label{eq:homo}
\Cl(X)_{\tors}\longrightarrow \Clsc(X,P)
\end{equation} 
is injective.
\end{scorollary}

\begin{proof}[Sketch of the proof]
Assume that $(X,C)$ is an extremal curve germ.
By Lemma~\ref{lemma:cov} there exists the \textit{maximal} finite cyclic cover $\tau_\flat:(X^\flat, C^\flat)\to (X,C)$
which is \'etale outside $\Sing(X)$ so that any other cover $\tau: (X',C')\to (X,C)$ which is \'etale outside $\Sing(X)$ 
is induced by $\tau_\flat$. Moreover, $\tau_\flat^{-1} (P)$ is a point for some $P\in C$. 
Any non-trivial element $D\in \Cl(X)_{\tors}$ defines a cover $\tau: (X',C')\to (X,C)$ as above, 
hence $\tau_\flat^{*}D\sim 0$ and $D$ is not Cartier at~$P$.
This implies that the group $\Clsc(X^\flat)$ is torsion free and the homomorphism~\eqref{eq:homo}
is injective. 
Hence $\Cl(X)_{\tors}$ is cyclic.
If $(X,C)$ is a flipped germ, then the assertion follows from the above
and the isomorphism~\eqref{eq:Cl=Cl}.
\end{proof}

\begin{scorollary}[cf. {\cite[Corollary~1.10]{Mori:flip}}]
\label{cor:Clsc-gen}
Let $(X,C)$ be either an extremal curve germ or a flipped germ, where $C$ is irreducible.
Assume that $(X,C)$ is primitive. Then $\Clsc(X)\simeq \ZZ$ and for 
a generator $D$ of $\Clsc(X)$ we have 
\[
D\cdot C=\frac{\pm 1}{\mathrm{index}(X)}.
\]
\end{scorollary}

\begin{scorollary}
\label{corImp}
Let $(X,C)$ be either an extremal curve germ or a flipped germ, where $C$ is irreducible.
Then the germ $(X,C)$ is imprimitive if and only if
one of the following folds:
\begin{enumerate}
\item 
\label{corImp:a}
$(X,C)$ has a locally imprimitive point $P$, or

\item
\label{corImp:b}
$(X,C)$ has two primitive points $P_1$ and $P_2$ of indices $m_1$ and $m_2$ with $\gcd(m_1,m_2)\neq 1$.
\end{enumerate}
If furthermore 
$(X,C)$ is an imprimitive extremal curve germ of
splitting degree $m$, then in the case 
\ref{corImp:a} the splitting degree of $C$ at~$P$ equals $m$ and
$(X,C)$ has no 
other non-Gorenstein points, and in the case 
\ref{corImp:b} we have $\gcd(m_1,m_2)=m$ and $\Sing(X)=\{P_1,\, P_2\}$.
\end{scorollary}

\begin{proof}
The division into cases follows from \cite[Corollary~1.12]{Mori:flip} 
and the assertions on singularities follow from \cite[Theorems~6.7 and 9.4]{Mori:flip} and 
\cite[Theorem~1.2]{MP:cb1}.
\end{proof}

\begin{sremark}
\label{rem:cla}
The case~\ref{corImp:a} occurs if and only if $(X,C)$ is of type~\typec{ID^\vee}, \typec{IE^\vee}, \typec{II^\vee} or \typec{IA^\vee} 
(the last is a part of \typec{k1A}).
In fact, the case \typec{k1A} can be primitive or imprimitive, and the 
imprimitive \typec{k1A} is equivalent to \typec{IA^\vee}.
The case~\ref{corImp:b} occurs if and only if $(X,C)$ is of type~\typec{k2A} with $\gcd(m_1,m_2)>1$.
\end{sremark}

\begin{scorollary}
\label{cor:divI}
Let $(X,C)$ be an extremal curve germ of  splitting degree $m\ge 1$.
\begin{enumerate}
\item
\label{cor:divIb}
If $(X,C)$ is a $\QQ$-conic bundle,
then $(Z,o)$ is a Du Val singularity of type \type{A_{m-1}}.
In particular, the point $(Z,o)$ is smooth if and only if
$\Cl(X)_\tors=0$. 

\item
\label{cor:divIa}
If  $(X,C)$ is divisorial and $C$ is irreducible, then 
the $f$-exceptional locus $E$ is
a $\QQ$-Cartier divisor and $(Z,o)$ is a terminal singularity of index~$m$.
In particular, the point $(Z,o)$ is Gorenstein if and only if $\Cl(X)_\tors=0$.
\end{enumerate}
\end{scorollary}
\begin{proof}
In view of Corollary~\xref{cor:Clsc} 
the assertion~\ref{cor:divIb} follows from \cite[Theorem~1.2.7]{MP:cb1} and
\ref{cor:divIa} follows from 
\cite[Theorem~3.1]{MP:IA}.
\end{proof}

\begin{slemma}
\label{lemma:flips2-2}
Let $(X,C)\dashrightarrow (X^+,C^+)$ be a flip of threefolds with terminal singularities 
with irreducible $C$ and $C^+$.
\begin{enumerate}
\item \label{lemma:flips2-2a}
The germ $(X,C)$ is primitive if and only if so $(X^+,C^+)$ is.
\item \label{lemma:flips2-2b}
Assume that $(X,C)$ is primitive.
Then
\begin{equation*}
\label{eq:KCindexa}
\operatorname{index}(X^+)\, (K_{X^+}\cdot C^+)= -\operatorname{index}(X)\, (K_X\cdot C).
\end{equation*} 
\end{enumerate} 
\end{slemma}

\begin{proof}
The assertion~\ref{lemma:flips2-2a} is obvious because $\Pic(X)\simeq \Pic(X^+)\simeq \ZZ$ and $\Clsc(X)\simeq \Clsc(X^+)$.
For the assertion~\ref{lemma:flips2-2b} we note that $K_X$ and $K_{X^+}$ generate subgroups of the same index in 
$\Clsc(X)$ and $\Clsc(X^+)$, respectively. Then~\ref{lemma:flips2-2b} 
follows from~\ref{lemma:flips2-2a} and Corollary~\ref{cor:Clsc-gen}.
\end{proof}

\begin{lemma}
\label{lemma:div}
Let $f: X\to X'$ be a divisorial $K$-negative extremal contraction of threefolds with terminal singularities
and let $E$ be the exceptional divisor.
Let $L\subset X$ be a proper curve such that $L\cap E\neq \varnothing$ and $L\not\subset E$, and let $L':=f(L)$. Then 
\[
K_X\cdot L\ge K_{X'}\cdot L'+\frac 1n,
\]
where $n$ is the index of $X$ near $E\cap L$.
\end{lemma}

\begin{proof}
Follows from the formula $K_X=f^*K_{X'}+E$ and the fact that the divisor
$nE$ is Cartier at the points of the set $E\cap L$ \cite[Lemma~5.1]{Kaw:Crep}.
\end{proof}

\begin{lemma}
\label{lemma:flip}
Let $\psi: X\dashrightarrow X^+$ be a flip
and let $L\subset X$ be a proper curve that meets the flipping locus but 
is not contained in it.
Let $L^+:=\psi_* L$. Then 
\[
K_X\cdot L>K_{X^+}\cdot L^+.
\]
\end{lemma}
\begin{proof}
Let 
\[
\xymatrix@R=1em{
&\hat X\ar[dr]^q\ar[dl]_p&
\\
X\ar@{-->}[rr]&&X^+
}
\]
be a common resolution
and let $\hat L\subset \hat X$
be the proper transform of $L$. We can write 
\[
K_{\hat X}=p^* K_X+\sum a_i E_i = q^* K_{X^+}+\sum a_i^+ E_i,
\]
where the $E_i$ are prime divisors. Note that all these divisors are exceptional for both $p$ and $q$. Then
\[
K_{\hat X}\cdot \hat L=K_X\cdot L+\sum a_i E_i\cdot \hat L = K_{X^+}\cdot L^++\sum a_i^+ E_i\cdot \hat L.
\]
Therefore,
\[
K_X\cdot L= K_{X^+}\cdot L^++\sum (a_i^+-a_i) E_i\cdot \hat L.
\]
It remains to note that $a_i^+\ge a_i$ for all $i$ (see \cite[2.13]{Shokurov:nV} or \cite[Proposition~5.1.11(3)]{KMM}).
\end{proof}

\subsection{}
In the table below we list some properties of all the irreducible 
extremal curve flipping germs, which are not of type \typec{k1A} nor 
\typec{k2A}.
Thus we let $(X,C)$ be a flipping extremal curve germ with irreducible $C$ and let $(X,C)\dashrightarrow(X^+,C^+)$
be the corresponding flip. 
\begin{table}[H]
\setlength{\tabcolsep}{12pt}
\renewcommand{\arraystretch}{1.7}
\begin{tabularx}{0.96\textwidth}{|l|l|l|l|X|}
\hline
type & \cite{KM92} & $K_X\cdot C$ & $K_{X^+}\cdot C^+$ & $\Sing^{\mathrm{nG}}(X^+)$
\\ \hline
\multirow{2}{*} {\typec{cD/3}} & A.1.2.1, A1.2.2 & \multirow{2}{*} {$-1/3$} & $1/2$ & one index $2$ point 
\\ \hhline{~-~--}
& A.1.2.3 & & 1 & $\varnothing$
\\ \hline
\multirow{3}{*} {\typec{ IIA}} & A.2.2.1 & \multirow{3}{*} {$-1/4$} & $1/6$ & index $2$ point and index $3$ point 
\\ \hhline{~-~--}
~ & A.2.2.3, A.2.2.3 & & $1/2$ & one index 2 point 
\\ \hhline{~-~--}
~ & A.2.2.4, A.2.2.5 & & $1/2$ & one index 2 point 
\\ \hline
\multirow{2}{*} {\typec{IC}} & A.3.2.1 & \multirow{2}{*} {$-1/m$} & $1/2$ & one index $2$ point 
\\ \hhline{~-~--}
& A.3.2.2. & & $1$ &$\varnothing$
\\ \hline
\typec{kAD} & A.4.3.2 & $-1/2m$ & $1/2$ & one index $2$ point 
\\\hline
\end{tabularx}
\caption{Flips that are not of type \typec{k1A} nor 
\typec{k2A}}
\label{table2}
\end{table}

\begin{cremark}
The flips and their singularities
are described in the Appendix of \cite{KM92} (see the second column).
In all cases we have
$-K_X\cdot C=1/\operatorname{index}(X)$
(the third column). 
This follows from the formula $-K_X\cdot C=1-\sum w_{P_i}(0)$ \cite[(2.3.2)]{Mori:flip},
computations of the invariant $w_{P}(0)$ \cite[Corollary~2.10]{Mori:flip},
and the local description of singularities of $(X,C)$ \cite[Appendix]{KM92}.
Then the value $K_{C^+}\cdot C^+$ in the fourth column follows from Lemma~\ref{lemma:flips2-2}
since in all cases the germ $(X,C)$ is primitive.
\end{cremark}
\def\cprime{$'$}

\end{document}